\documentclass[a4ppaper,14pt]{amsart}

\usepackage{enumerate}
\usepackage[latin1]{inputenc}
\usepackage{amsmath}
\usepackage{amsthm}
\usepackage{latexsym}
\usepackage{amssymb}
\usepackage{amsmath}
\usepackage{comment}
\usepackage[all]{xy}
\usepackage{calc}
\usepackage{multicol}
\usepackage{slashed}
\usepackage{color}

\newtheorem{thm}{Theorem}[section]
\newtheorem{prop}{Proposition}[section]
\newtheorem{defi}{Definition}[section]
\newtheorem{lem}{Lemma}[section]

\newtheorem{rem}{Remark}[section]
\theoremstyle{notation}

\newcommand{\R}{\mathbb{R}}

\numberwithin{equation}{section}
\newcommand{\N}{\mathbb{N}}

\newcommand{\eps}{\epsilon}

\newcommand{\wto}{\rightharpoonup}

\newcommand{\vertiii}[1]{{\left\vert\kern-0.25ex\left\vert\kern-0.25ex\left\vert #1
\right\vert\kern-0.25ex\right\vert\kern-0.25ex\right\vert}}

\usepackage[textwidth=138mm,
textheight=215mm,
left=30mm,
right=30mm,
top=25.4mm,
bottom=25.4mm,
headheight=2.17cm,
headsep=4mm,
footskip=12mm,
heightrounded,]{geometry}

\usepackage[colorlinks=true,urlcolor=blue,
citecolor=black,linkcolor=black,linktocpage,pdfpagelabels,
bookmarksnumbered,bookmarksopen]{hyperref}

\makeatletter
\newcommand{\leqnomode}{\tagsleft@true}
\newcommand{\reqnomode}{\tagsleft@false}
\makeatother

\definecolor{green}{rgb}{0.25,0.75,0}

\begin{document}

\reqnomode

\title{Normalized solutions to nonlinear Schr\"odinger equations with competing Hartree-type nonlinearities}

\author[D. Bhimani, T. Gou, H. Hajaiej]{Divyang Bhimani, Tianxiang Gou, Hichem Hajaiej}

\address{Divyang Bhimani,
\newline \indent Department of Mathematics, Indian Institute of Science Education and Research,
\newline \indent Dr. Homi Bhabha Road, Pune 411008, India.}
\email{divyang.bhimani@iiserpune.ac.in}

\address{Tianxiang Gou
\newline \indent School of Mathematics and Statistics, Xi'an Jiaotong University,
\newline \indent Xi'an, Shaanxi 710049, People's Republic of China.}
\email{tianxiang.gou@xjtu.edu.cn}

\address{Hajaiej Hichem, 
\newline \indent Department of Mathematics, College of Natural Science State University, 
\newline \indent 5151 State Drive, 90032 Los Angeles, California, USA.}
\email{hhajaie@calstatela.edu}

\begin{abstract} 
In this paper, we consider solutions to the following nonlinear Schr\"odinger equation with competing Hartree-type nonlinearities,
$$ 
-\Delta u + \lambda u=\left(|x|^{-\gamma_1} \ast |u|^2\right) u - \left(|x|^{-\gamma_2} \ast |u|^2\right) u\quad \mbox{in} \,\, \R^N,
$$
under the $L^2$-norm constraint
$$
\int_{\R^N} |u|^2 \, dx=c>0,
$$
where $N \geq 1$, $0<\gamma_2 < \gamma_1 <\min\{N, 4\}$ and $\lambda \in \R$ appearing as Lagrange multiplier is unknown. First we establish the existence of ground states in the mass subcritical, critical and supercritical cases. Then we consider the well-posedness and dynamical behaviors of solutions to the Cauchy problem for the associated time-dependent equations.



\end{abstract}

\maketitle

\section{Introduction}

In this paper, we are concerned with solutions to the following nonlinear Schr\"odinger equation with competing Hartree-type nonlinearities,
\begin{align} \label{equ}
-\Delta u + \lambda u=\left(|x|^{-\gamma_1} \ast |u|^2\right) u - \left(|x|^{-\gamma_2} \ast |u|^2\right) u\quad \mbox{in} \,\, \R^N,
\end{align}
under the $L^2$-norm constraint
\begin{align} \label{mass}
\int_{\R^N} |u|^2 \, dx=c>0,
\end{align}
where $N \geq 1$, $0<\gamma_2 < \gamma_1 <\min\{N, 4\}$ and $\lambda \in \R$ appearing as Lagrange multiplier is unknown. The problem under consideration comes from the study of standing waves to the following time-dependent equation,
\begin{align}\label{nlscn}
\begin{cases}
\textnormal{i} \psi_t =\Delta \psi +\left(|x|^{-\gamma_1} \ast |\psi|^2\right) \psi - \left(|x|^{-\gamma_2} \ast |\psi|^2\right) \psi,\\
\psi(x,0)=\psi_0.
\end{cases}
\end{align}
Here a standing wave to \eqref{nlscn} is a solution having the form
$$
\psi(t,x)=e^{-\textnormal{i}\lambda t} u(x), \quad \lambda \in \R.
$$
Clearly, a standing wave $\psi$ is a solution to \eqref{nlscn} if and only if $u$ is a solution to \eqref{equ}. Physically, the equation admits many applications in mean field models for binary mixtures of Bose-Einstein condensates or for binary gases of fermion atoms in degenerate quantum states, see \cite{BFKMM, EGBB, Ma}. Let $W(x)=C_{\gamma_1}|x|^{-\gamma_1} +C_{\gamma_2}|x|^{-\gamma_2}$. In physics, when $\gamma_1, \gamma_2 >1$, then $W$ is the van der Waals type potential, that is, two-body potentials of short range, see \cite{DLP, YLT, ZN}. The van der Waals coefficient $C_6, C_8$ and $C_{10}$ of alkaline-earth interactions calculated by Porsev and Derevianko using relativistic many-body perturbation theory are believed to be accurate to $1/100$, see \cite{PD}. In the present paper, by scaling argument, we shall take $C_{\gamma_1} =1 $ and $C_{\gamma_2}=-1$.

Note that $L^2$-norm of any solution to \eqref{nlscn} is conserved along time, i.e. for any $t \in [0, T)$,
$$
\|\psi(t)\|_2=\|\psi_0\|_2.
$$
In this sense, it is interesting to consider standing waves to \eqref{nlscn} having prescribed $L^2$-norm. This leads to the study of solutions to \eqref{equ}-\eqref{mass}. Such solutions are often referred to as normalized solutions, which correspond to critical points of the underlying energy functional $E$ restricted to $S(c)$, where
$$
E(u):=\frac 12 \int_{\R^N}|\nabla u|^2 \,dx + \frac 14 \int_{\R^N} \int_{\R^N}\frac{|u(x)|^2|u(y)|^2}{|x-y|^{\gamma_2}} \, dxdy-\frac 14 \int_{\R^N} \int_{\R^N}\frac{|u(x)|^2|u(y)|^2}{|x-y|^{\gamma_1}} \, dxdy.
$$
and
$$
S(c):=\left\{u \in H^1(\R^N) : \int_{\R^N} |u|^2 \, dx=c\right\}.
$$
In this situation, the parameter $\lambda$ is unknown and will be decided as Lagrange multiplier related to the constraint.

In the mass subcritical case, the study of normalized solutions benefits from the well-known Lions concentration compactness principle in \cite{Li1, Li2}. In this case, the associated energy functional restricted on constraint is bounded from below. Then normalized solutions can be achieved as global minimizers related to the minimization problems, whose existence is a direct consequence of the use of Lions concentration compactness principle. It was first adapted in \cite{CP} to prove orbital stability of standing waves to nonlinear Schr\"odinger equations. Following these initial works, the authors in \cite{Sh} and \cite{BS, CDSS} studied the existence of standing waves to nonlinear Schr\"odinger equations and nonlinear Schr\"odinger-Poission systems, respectively. Furthermore, the authors in \cite{AB, NW1, NW2, NW3} considered orbital stability and existence of standing waves to nonlinear Schr\"odinger systems in the one dimensional case. The results was extended to the higher dimensional case in \cite{GJ2}, see also \cite{G, GG}. We also refer the readers to \cite{Gou}, where orbital stability and existence of standing waves to nonlinear Schr\"odinger systems with a partial confinement were investigated. And the existence of ground states for the spinor Bose-Einstein condensates in the one dimensional was established in \cite{CCW}. However, in the mass critical and supercritical cases, the associated energy functional restricted on constraint becomes unbounded from below. In this situation, normalized solutions often correspond to saddle type critical points or local minimizers, whose existence is based on minimax arguments. Following the early work due to Jeanjean \cite{Je}, there exist many researchers who are devoted to the study in this direction. Regarding the consideration of normalized solutions to nonlinear Schr\"odinger equations, we first refer to the readers to \cite{BFG, S1} and \cite{JL, S2}, where the existence of normalized solutions to nonlinear Schr\"odinger equations with combined local nonlinearities in the Sobolev subcritical and critical cases was detected, respectively. The existence of normalized solutions to nonlinear Schr\"odinger equations with general nonlinearities was studied in \cite{BV, JS}. Moreover, the authors in \cite{BJT,JL1} established the existence of normalized solutions to nonlinear Schr\"odinger-Poission systems in the Sobolev subcritical and critical cases in $\R^3$, respectively. For the study of normalized solutions to nonlinear Schr\"odinger-Poission systems in $\R^2$, we refer the readers to \cite{CJ}. We also allude the readers to \cite{BJ, BCGJ, CGH, GZ, LY}, where normalized solutions to other types of nonlinear Schr\"odinger equations (for example mixed dispersive nonlinear Schr\"odinger equations and Chern-Simons-Schr\"odinger equations) were considered. Regarding the study of normalized solutions to nonlinear Schr\"odinger systems, we refer the readers to \cite{BS1, BS2, BJS, BZZ}, where the authors proved the existence of normalized solutions to nonlinear Schr\"odinger system in $\R^3$. For the existence of normalized solutions to nonlinear Schr\"odinger system in $\R^N$ for $N \geq 1$, see \cite{GJ1, LWYZ}. 

The study carried out in the present paper is mainly inspired by \cite{CJL}, where the authors considered normalized solutions to the following equation,
\begin{align} \label{equ1111}
-\Delta u + \lambda u=\left(|x|^{-\gamma_1} \ast |u|^2\right) u + \left(|x|^{-\gamma_2} \ast |u|^2\right) u\quad \mbox{in} \,\, \R^N.
\end{align}
As an extension, the aim of the paper is to completely investigate normalized solutions to \eqref{equ} with competing Hartree nonlinearities in the mass subcritical, critical and supercritical cases. Very recently, the authors in \cite{JLuo} studied the existence and asymptotic behaviors of normalized solutions to \eqref{equ1111} for $\gamma_1=4$ and also established nonexistence of normalized solutions to \eqref{equ} for $\gamma_1=4$. It is worth mentioning \cite{YY}, where the existence, multiplicity and nonexistence of solutions to \eqref{equ} with $\lambda>0$ fixed (without $L^2$-norm constraint) were established in $\R^3$. Here we are concerned with normalized solutions to nonlinear Schr\"odinger equations with completing Hartree-type (nonlocal) nonlinearities in the Sobolev subcritical case,  we refer the readers to \cite{BFG, S1} for the study of the problems with competing local nonlinearities in the Sobolev subcritical case.

In order to present our main results, we introduce a useful scaling of $u$ as
$$
u_t(x):=t^{\frac{N}{2}}u(tx), \quad t>0, \ x \in \R^N.
$$
Direct computations imply that $\|u_t\|_2=\|u\|_2$ and
\begin{align} \label{scaling}
\hspace{-1.55cm}E(u_t):=\frac {t^2}{2} \int_{\R^N}|\nabla u|^2 \,dx + \frac {t^{\gamma_2}}{4}\int_{\R^N} \int_{\R^N}\frac{|u(x)|^2|u(y)|^2}{|x-y|^{\gamma_2}} \, dxdy-\frac {t^{\gamma_1}}{4} \int_{\R^N} \int_{\R^N}\frac{|u(x)|^2|u(y)|^2}{|x-y|^{\gamma_1}} \, dxdy.
\end{align}
Such a scaling will be frequently employed and play an important role in the discussion of the paper. In the sense of Gagliardo-Nirenberg' inequality \eqref{inequ}, we shall refer $\gamma =2$ to the mass critical case, $\gamma <2$ and $\gamma>2$ to the mass subcritical and supercritical cases, respectively. To ensure that the nonlocal integral terms appearing in the functional $E$ with respect to the Hartree-type nonlinearities are well-defined in $H^1(\R^N)$, we also need to assume that $0<\gamma_1, \gamma_2<N$.

First of all, we investigate solutions to \eqref{equ}-\eqref{mass} in the mass subcritical case. In this case, we have that $0<\gamma_2<\gamma_1<\min\{N, 2\}$. Then, by Gagliardo-Nirenberg's inequality \eqref{inequ}, we easily find that $E$ restricted on $S(c)$ is bounded from below. Then we are able to introduce the following minimization problem
\begin{align} \label{gmin}
m(c):=\inf_{u \in S(c)} E(u).
\end{align}
Obviously, any minimizer to \eqref{gmin} is a solution to \eqref{equ}-\eqref{mass}. In this case, our first result reads as follows.

\begin{prop} \label{prop}
Let $0<\gamma_2<\gamma_1<\min\{N, 2\}$. Then the following assertions hold.
\begin{itemize}
\item [$(\textnormal{i})$] $-\infty<m(c) \leq 0$ for any $c \geq 0$.
\item [$(\textnormal{ii})$]  The function $c \mapsto m(c)$ is continuous for any $c >0$.
\item [$(\textnormal{iii})$] $m(c) \leq m(c_1) +m(c_2)$ for any $c_1, c_2 \geq 0$ and $c=c_1+c_2>0$. In addition, the function $c \mapsto m(c)$ is nonincreasing on $(0, \infty)$.
\item[$(\textnormal{iv})$] $\lim_{c \to 0^+} m(c)=0$ and $\lim_{c \to \infty}m(c)=-\infty$.
\end{itemize}
\end{prop}
 
In the mass subcritical case, there do exist global minimizers to the minimization problem as the underlying energy functional related to \eqref{equ1111} restricted on $S(c)$ for any $c>0$. However, this is no longer a case for the problem under our consideration.

\begin{thm} \label{thmsub1}
Let $0<\gamma_2<\gamma_1<\min\{N, 2\}$ and $N \geq 3$, there exists a constant $\hat{c}_0>0$ depending only on $\gamma_1, \gamma_2$ and $N$ such that, for any $0<c<\hat{c}_0$, $m(c)=0$. In addition, for any $0<c<\hat{c}_0$, there exists no minimizers to \eqref{gmin}.
\end{thm}

In the spirit of the Lions concentration compactness principle in \cite{Li1, Li2}, we also have the following result.

\begin{thm}  \label{thmsub2}
Let $0<\gamma_2<\gamma_1<\min\{N, 2\}$. Then there exists a constant $\tilde{c}_0>0$ depending only on $\gamma_1, \gamma_2$ and $N$ such that, for any $c>\tilde{c}_0$, $m(c)<0$ and the following strict subadditivity inequality holds,
\begin{align} \label{ssub}
m(c)<m(c_1) + m(c_2),
\end{align}
where $c=c_1+c_2$ and $0<c_1, c_2<c$. In addition, for any $c>\tilde{c}_0$, any minimizing sequence to \eqref{gmin} is compact in $H^1(\R^N)$ up to translations and there exists a minimizer to \eqref{gmin}.
\end{thm}

Second we study solutions to \eqref{equ}-\eqref{mass} in the mass critical case. In this case, we first have the following result.

\begin{thm}\label{nonexistence1}
Let $0<\gamma_2<\gamma_1=2$ and $N \geq 3$. Then there exists a constant $\tilde{c}_1>0$ depending only on $N$ such that $m(c)=0$ for any $0<c\leq \tilde{c}_1$ and $m(c)=-\infty$ for any $c>\tilde{c}_1$. In addition, there exists no solutions to \eqref{equ}-\eqref{mass} for any $0<c \leq \tilde{c}_1$.
\end{thm}

Third we consider solutions to \eqref{equ}-\eqref{mass} in the mass supercritical case.  In this case, we have that $\gamma_1>2$. It then follows from \eqref{scaling} that $E(u_t) \to -\infty$ as $t \to \infty$. Hence there holds that $m(c)=-\infty$ for any $c>0$. For this reason, it is impossible to obtain solutions to \eqref{equ}-\eqref{mass} with the help of the global minimization problem \eqref{gmin}. In this situation, we introduce the following minimization problem,
\begin{align} \label{minmin}
\Gamma(c):=\inf_{u \in P(c)} E(u),
\end{align}
where $P(c)$ is the related Pohozaev manifold defined by
$$
P(c):=\{u \in S(c) : Q(u)=0\},
$$
and
\begin{align*}
Q(u):&=\frac{\partial}{\partial t} E(u_t) \mid_{t=1}\\
&=\int_{\R^N}|\nabla u|^2 \,dx + \frac {\gamma_2}{4}\int_{\R^N} \int_{\R^N}\frac{|u(x)|^2|u(y)|^2}{|x-y|^{\gamma_2}} \, dxdy-\frac{\gamma_1}{4} \int_{\R^N} \int_{\R^N}\frac{|u(x)|^2|u(y)|^2}{|x-y|^{\gamma_1}} \, dxdy.
\end{align*}
Here $Q(u)=0$ is the Pohozaev identity related to \eqref{equ}-\eqref{mass}, see Lemma \ref{ph}. The existence of solutions to \eqref{equ}-\eqref{mass} is derived by investigating the minimization problem \eqref{minmin}.

\begin{thm}\label{thm1}
Let $0<\gamma_2<\gamma_1<\min\{N, 4\}$, $\gamma_1>2$ and $N \geq 3$. Then there exists a constant $\hat{c}_1>0$ such that, for any $0<c<\hat{c}_1$, \eqref{equ}-\eqref{mass} has a ground state $u_c \in S(c)$ satisfying $E(u_c)=\Gamma(c)$.
\end{thm}

To prove this theorem, the essential argument is to show that there exists a Palais-Smale sequence for $E$ restricted on $S(c)$ at the level $\Gamma(c)$. Later, by mainly using the coerciveness of $E$ restricted on $P(c)$, the nonincreasing property of $\Gamma(c)$ on $(0, +\infty)$ as well as positiveness of the associated Lagrange multiplier, we can complete the proof.

\begin{rem}
In Theorem \ref{thm1}, the condition that $c>0$ is small is applied to guarantee that the associated Lagrange multiplier is positive, see Lemma \ref{sign}. It is hard to identify the sign of the Lagrange multiplier for any $c>0$. Then it is unknown to us that the existence of solutions to \eqref{equ}-\eqref{mass} remains true for any $c>0$. This is different from the problem with one Hartree-type nonlinearity treated in \cite{Luo}, where the associated Lagrange multiplier is positive for any $c>0$, see \cite[Lemma 2.7]{Luo}.
\end{rem}

Let us now come back to the mass critical case. In Theorem \ref{nonexistence1}, we note that there exists no solutions to \eqref{equ}-\eqref{mass} for any $0<c<\tilde{c}$. It would be interesting to ask if there exist solutions to \eqref{equ}-\eqref{mass} for $c>\tilde{c}_1$. Indeed, there do exist solutions to \eqref{equ}-\eqref{mass} for some $c>\tilde{c}_1$.

\begin{thm} \label{thmcritical}
Let $0<\gamma_2<\gamma_1=2$ and $N \geq 3$. Then there exists a constant $c_1^*>\tilde{c}_1$ such that, for any $\tilde{c}_1<c<c_1^*$, there exists a ground state $u_c \in S(c)$ to \eqref{equ}-\eqref{mass} satisfying $E(u_c)=\Gamma(c)$.
\end{thm}

To prove Theorem \ref{thmcritical}, we only need to replace the roles of $S(c)$ in the proof of Theorem \ref{thm1} by $\mathcal{S}(c)$, where
$$
\mathcal{S}(c):=\left\{ u \in S(c) : \int_{\R^N} |\nabla u|^2 \, dx <\frac 12\int_{\R^N} \int_{\R^N}\frac{|u(x)|^2|u(y)|^2}{|x-y|^{2}} \, dxdy \right\}.
$$
This completes the proof.

Next we reveal some asymptotical properties of the function $c \mapsto \Gamma(c)$ as $c \to 0^+$ and $c \to \infty$, respectively.

\begin{prop} \label{asym}
Let $0<\gamma_2<\gamma_1<\min\{N, 4\}$, $\gamma_1>2$ and $N \geq 3$. Then $\lim_{c \to 0^+}\Gamma(c)=+ \infty$ and $\lim_{c \to \infty} \Gamma(c)>0$. If in addition $N \geq 5$, then there exists a constant $c_{\infty}>0$ such that, for any $c \geq c_{\infty}$, $\Gamma(c)=m$, where $m>0$ is the least energy level of solutions to \eqref{equz}.
\end{prop}

To discuss asymptotical behavior of the function $c \mapsto \Gamma(c)$ as $c \to \infty$, we need to study the following zero mass equation
\begin{align} \label{equz}
-\Delta u =\left(|x|^{-\gamma_1} \ast |u|^{2}\right) u - \left(|x|^{-\gamma_2} \ast |u|^2\right) u\quad \mbox{in} \,\, \R^N,
\end{align}
where $0<\gamma_2<\gamma_1<\min\{N, 4\}$. 
To look for solutions to \eqref{equz}, we shall work in the underlying Sobolev space $X$ defined by
$$
X:=\left\{ u\in D^{1,2}(\R^N) : \int_{\R^N}\frac{|u(x)|^2|u(y)|^2}{|x-y|^{\gamma_2}} \, dxdy<\infty \right\}
$$
equipped with the norm 
$$
\|u\|_X:=\left(\int_{\R^N} |\nabla u|^2 \, dx + \left(\int_{\R^N}\frac{|u(x)|^2|u(y)|^2}{|x-y|^{\gamma_2}} \, dxdy \right)^{\frac 12} \right)^{\frac 12}.
$$

\begin{prop} \label{l2}
Let $0<\gamma_2<\gamma_1<4$, $\gamma_1>2$ and $N \geq 5$ and let $u \in X$ be a solution to \eqref{equz}. Then $u \in C^2(\R^N)$ and 
$$
u(x) \sim |x|^{-\frac{N-1}{2}+\frac{\gamma_2}{4}} e^{-\frac{4}{2-\gamma_2} |x|^{1-\frac{\gamma_2}{2}}} \quad \mbox{if} \,\,\, 0<\gamma_2<2, \quad u(x) \sim |x|^{2-N} \quad \mbox{if} \,\,\, 2 \leq \gamma_2<4 \quad \mbox{as} \,\,\, |x| \to \infty.
$$ 
In particular, there holds that $u \in L^2(\R^N)$.
\end{prop}

Next we investigate well-posedness of solutions to \eqref{nlscn}. In this direction, we first have the following result.

\begin{thm}\label{LW} 
Let $N \geq 1$, $0<\gamma_2<\gamma_1<N$ and $\psi_0 \in H^{s}(\R^N)$ with $s\geq \frac{\gamma_1}{2}$ and $s \geq 1$. Then there exists a positive time $T$ such that  \eqref{nlscn} has a unique solution with  $\psi \in C([0, T], H^s(\R^N))$ with $\|u\|_{L^{\infty}_T H^s} \lesssim \|\psi_0\|_{H^s}$. In addition, there holds conservation of mass and energy, i.e for any $ t \in [0, T)$,
$$
\|\psi(t)\|_{L^2}= \|\psi_0\|_{L^2}, \quad E(\psi(t))=E(\psi_0).
$$
\end{thm}

As a consequence of Theorem \ref{LW}, we can further derive the global existence of solutions to \eqref{nlscn}.

\begin{thm} \label{GLW}
Let $N \geq 3$, $0<\gamma_1<\gamma_2<2$ and $s\geq 1$. Let $T^*$ be the maximum existence time of the solution $\psi$ in Theorem \ref{LW}. Then $T^*= \infty.$ In addition, there holds that $\|\psi(t)\|_{H^1} \leq C_0$ for any $t \geq 0$,
where $C_0>0$ depends on $E(u_0)$ and $\|u_0\|_2$.
\end{thm}

For relatively rough data in $\psi_0 \in H^{s}(\R^N)$ with $s<\frac{\gamma_1}{2}$, we have the following well-posedness result.

\begin{thm} \label{LW1}
Let  $0<\gamma_2< \gamma_1<N$ and $\psi_0 \in H^s(\R^N)$ with $ \max \{0, \frac{\gamma_1}{2}-1\}\leq s < \frac{\gamma_2}{2}.$ Then there exists a  positive time $T$ such that \eqref{nlscn} has a unique solution $$\psi\in C([0,T], H^s(\R^N)) \cap L^{q_1}_{T}(H^s_{r_1}(\R^N))\cap L^{q_2}_{T}(H^s_{r_2}(\R^N)),$$
where $(q_i, r_i)$ are defined by \eqref{yu} and $H^{s}_r(\R^N)=(I-\Delta)^{-s/2}L^r(\R^N).$
\end{thm}

We now show global well-posedness and scattering of solutions to \eqref{nlscn} for small initial data.  The method of proof is inspired by \cite{Miao}.  Note that energy scattering for the classical Hartree equation goes back to Ginibre and Velo in  \cite{GV}.

\begin{thm} \label{sc}
Let $N \geq 3$, $2<\gamma_2< \gamma_1 <N$ and $s\geq s_{\gamma_1}=\frac{\gamma_1}{2}-1.$ Then there exists $\rho>0$ such that for any $\psi_0\in H^{s}(\R^N)$ with $\|\psi_0\|_{\dot{H}^{s_{\gamma_1}}} \leq \rho,$ \eqref{nlscn} has a unique solution $\psi \in (C \cap L^{\infty})(\mathbb R, H^s(\R^N)) \cap L^4 (\mathbb R, H^s_{\frac{2N}{N-1}}(\R^N)).$ In addition, there exists $\psi^+ \in H^s(\R^N)$ such that
\[ \|\psi(t)-U(t)\psi^+\|_{H^s} \to 0 \quad \text{as} \ t\to \infty,\]
where $U(t)\psi^+:=e^{-\textnormal{i} \Delta t} \psi^+.$
\end{thm}

Next we present the global well-posedness of solutions to \eqref{nlscn} in $L^2(\R^N)$. 

\begin{thm}\label{miD}
Let $0<\gamma_2<\gamma_1<N$   and $0<\gamma_1 < \text{min} \{2, N\}.$ If $\psi_{0}\in L^{2}(\mathbb R^{N}),$ then \eqref{nlscn} has a unique global solution
$$\psi\in C(\mathbb R, L^{2}(\mathbb R^N))\cap L^{8/\gamma_1}_{loc}(\mathbb R, L^{4N/(2N-\gamma_1)} (\mathbb R^N)) \cap L^{8/\gamma_2}_{loc}(\mathbb R, L^{4N/(2N-\gamma_2)} (\mathbb R^N)).$$
In addition, there holds that conservation of mass, i.e. 
$$\|\psi(t)\|_{L^{2}}=\|\psi_{0}\|_{L^{2}}, \quad \forall \, t \in \mathbb R,$$
and for all  admissible pairs  $(p,q), \psi \in L_{loc}^{p}(\mathbb R, L^{q}(\mathbb R^N))$, where $(p, q)$ is admissible if $q\geq 2, r\geq 2$ and
$$\frac{2}{q} =  N \left( \frac{1}{2} - \frac{1}{r} \right), \quad (q,r,N)\neq(\infty,2,2).$$
\end{thm}

Finally, we give a blow-up theory for \eqref{nlscn} when $E(u_0)<0$, whose proof is inspired by Glassy \cite{RTG},  Du et al. \cite{DWZ} and Zheng \cite{JZ}.

\begin{thm}\label{mtbup}
Let $N\geq 3, 2 <\gamma_2<\gamma_1<\min\{N, 4\}$ and $u_0\in H^1(\R^N)$ with $E(u_0)<0.$ Let $\psi$ be a solution of \eqref{nlscn} on the maximal interval $(-T_{\min},  T_{\max}).$ Then there holds that
\begin{itemize}
\item [$(\textnormal{i})$] If $T_{\max} < \infty,$ then $\lim_{t\to T_{\max}} \| \psi(t)\|_{H^1}= \infty.$
\item  [$(\textnormal{ii})$] If  $T_{\max}= \infty,$ then there exists a sequence $t_n\to \infty$ such that
\[ \lim_{n\to \infty} \|\psi(t_n)\|_{H^1}= \infty.\]
\end{itemize}
\end{thm}

\begin{rem} 
It would be interesting to point out that the existence of normalized solutions to \eqref{equ} can be extended to the following nonlinear Schr\"odinger equation with general competing Hartree-type nonlinearities,
\begin{align*} 
-\Delta u + \lambda u=\left(|x|^{-\gamma_1} \ast |u|^{p}\right) |u|^{p-2} u - \left(|x|^{-\gamma_2} \ast |u|^p\right) |u|^{p-2} u\quad \mbox{in} \,\, \R^N,
\end{align*}
where $N \geq 1$, $0<\gamma_2 < \gamma_1 <N$, $\frac{2N-\gamma_2}{N}<p<\frac{2N-\gamma_1}{(N-2)^+}$ and $\lambda \in \R$ appearing as Lagrange multiplier is unknown.
\end{rem}

This paper is organized as follows. In Section \ref{section2}, we present some preliminary results used to establish our main results. In Section \ref{section3}, we consider the mass subcritical case and give the proofs of Theorems \ref{thmsub1} and \ref{thmsub2}. In Section \ref{section4}, we investigate the mass supercritical case and prove Theorem \ref{thm1}. Later, in Section \ref{section5}, we study the mass critical case and present the proofs of Theorems \ref{nonexistence1} and \ref{thmcritical}. Section \ref{section6} is devoted to the proofs of Propositions \ref{asym} and \ref{l2}. Finally, in Section \ref{section7}, we concern the well-posedness and dynamical behaviors of solutions to \eqref{nlscn} and prove Theorems \ref{LW}-\ref{mtbup}.
\medskip

{\noindent{\bf Acknowledgements.} T. Gou was supported by the National Natural Science Foundation of China (No.12101483) and the Postdoctoral Science Foundation of China (No.2021M702620). The authors would like to thank warmly the knowledgeable referees for the careful reading of our manuscript and for giving constructive comments and suggestions.}

\section{Preliminaries} \label{section2}

In this section, we are going to present some preliminary results used to prove our main theorems. Let us first show the well-known Hardy-Littlewood-Sobolev inequality in \cite{LL} and Gagliardo-Nirenberg inequality in \cite{Wei}.
\begin{lem} 
Let $p, r>1$, $0<\gamma<N$ and $\frac 1 p+\frac \gamma N + \frac 1 r=2$. Then there exists a constant $C_{N,p,\gamma}>0$ such that
\begin{align} \label{HLS}
\left|\int_{\R^N} \int_{\R^N} \frac{f(x)h(x)}{|x-y|^{\gamma}} \, dxdy \right| \leq C_{N, p, \gamma} \|f\|_p \|h\|_r.
\end{align}
\end{lem}

\begin{lem}
Let $2<p< \infty$ if $N=1,2$ and $2<p<2^*$ if $N \geq 3$. Then there exists a constant $C_{N,p}>0$ such that
\begin{align} \label{GN}
\int_{\R^N} |u|^p \,dx \leq C_{N, p} \left(\int_{\R^N} |\nabla u|^2 \, dx \right)^{\frac{N(p-2)}{4}}\left(\int_{\R^N}|u|^2 \,dx \right)^{\frac p 2 -\frac{N(p-2)}{4}},
\end{align}
where
\begin{align} \label{const1}
C_{N, p}:=\left(\frac{2p-(p-2)N}{(p-2)N}\right)^{\frac{N(p-2)}{4}} \frac{2p}{\left(2p-(p-2)N\right)\|R\|_2^{p-2}},
\end{align}
where $R$ is the ground state of the equation
\begin{align} \label{equ1}
-\Delta R + R=|R|^{p-2} R \quad \mbox{in} \,\, \R^N.
\end{align}
\end{lem}

\begin{lem}  (\cite{Ye}) Let $0<\gamma<\min\{N, 4\}$. Then there exists a constant $C_{N, \gamma}>0$ such that
\begin{align} \label{inequ}
\hspace{-1cm}\int_{\R^N} \int_{\R^N}\frac{|u(x)|^2|u(y)|^2}{|x-y|^{\gamma}} \, dxdy \leq C_{N,\gamma} \left(\int_{\R^N} |\nabla u|^2 \, dx\right)^{\frac{\gamma}{2}} \left(\int_{\R^N} |u|^2 \, dx\right)^{\frac{4-\gamma}{2}},
\end{align}
where
\begin{align} \label{const2}
C_{N, \gamma}:=\left(\frac{4-\gamma}{\gamma}\right)^{\frac {\gamma}{2}}\frac{4}{(4-\gamma)\|Q_{\gamma}\|_2^2},
\end{align} 
where $Q_{\gamma}$ is a ground state of the equation
\begin{align} \label{equ2}
-\Delta Q_{\gamma} + Q_{\gamma} =\left(|x|^{-\gamma} \ast |Q_{\gamma}|^2\right) Q_{\gamma} \quad \mbox{in} \,\, \R^N.
\end{align}
\end{lem}

\begin{lem}\label{ph}
Let $0<\gamma_2<\gamma_1<\min\{N, 4\}$ and let $u \in H^1(\R^N)$ be a solution to \eqref{equ}-\eqref{mass}. Then $Q(u)=0$.
\end{lem}
\begin{proof}
Let $\psi \in C_0^{\infty}(\R^N, [0, 1])$ be a cut-off function such that $\psi(x)=1$ for $x \in B_1(0)$, $\psi(x)=0$ for $x \notin B_2(0)$ and $|\nabla \psi(x)| \leq 2$ for any $x \in \R^N$. Define $\psi_{\sigma}(x):=\psi(\sigma x) \left(x \cdot \nabla u(x)\right)$ for $x \in \R^N$.
Multiplying \eqref{equ} by $\psi_{\sigma}$ and integrating on $\R^N$, we have that
\begin{align} \label{ph1}
\begin{split}
\hspace{-1cm}-\int_{\R^N} \Delta u \psi_{\sigma} \, dx + \lambda \int_{\R^N} u \psi_{\sigma} \, dx  \, dx =\int_{\R^N}\left(|x|^{-\gamma_1} \ast |u|^2\right) u \psi_{\sigma} \, dx- \int_{\R^N}\left(|x|^{-\gamma_2} \ast |u|^2\right) u \psi_{\sigma} \, dx.
\end{split}
\end{align}
Let us first compute the terms in the left side hand of \eqref{ph1}. By simple calculations, it is simple to see that
\begin{align*}
-\int_{\R^N} \Delta u \psi_{\sigma} \, dx&= \int_{\R^N} \nabla u \cdot \nabla \left(\psi(\sigma x) \left(x \cdot \nabla u(x)\right)\right) \, dx = \int_{\R^N} \psi(\sigma x) \left(|\nabla u|^2+ x \cdot \nabla \left(\frac{|\nabla u|^2}{2}\right)\right) \, dx\\
&=-\frac 12 \int_{\R^N}|\nabla u|^2 \left((N-2) \psi(\sigma x) + \sigma x \cdot \nabla \psi(\sigma x)\right) \, dx
\end{align*}
and
$$
\lambda \int_{\R^N} u \psi_{\sigma} \, dx =\frac {\lambda}{2} \int_{\R^N} \psi(\sigma x)\left(x \cdot \nabla \left(|u|^2\right) \right) \, dx =-\frac {\lambda}{2}  \int_{\R^N} |u|^2 \left(N \psi(\sigma x) + \sigma x \cdot \nabla \psi(\sigma x)\right) \,dx.
$$
Let us now compute the terms in the right side hand of \eqref{ph1}. Observe that
\begin{align*}
&\int_{\R^N}\left(|x|^{-\gamma_i} \ast |u|^2\right) u \psi_{\sigma} \, dx = \frac 12 \int_{\R^N} \int_{\R^N} \frac{|u(y)|^2}{|x-y|^{\gamma_i}} \left(\psi(\sigma x) x \cdot \nabla \left(|u|^2\right)\right) \, dx dy \\
&= \frac 14 \int_{\R^N} \int_{\R^N} \frac{|u(y)|^2}{|x-y|^{\gamma_i}} \left(\psi(\sigma x) x \cdot \nabla \left(|u|^2\right)\right) \, dx dy+\frac 14 \int_{\R^N} \int_{\R^N} \frac{|u(x)|^2}{|x-y|^{\gamma_i}} \left(\psi(\sigma y) y \cdot \nabla \left(|u|^2\right) \right) \, dx dy \\ 
&=-\frac 12 \int_{\R^N} \int_{\R^N} \frac{|u(x)|^2 |u(y)|^2}{|x-y|^{\gamma_i}} \left(N \psi(\sigma x) + \sigma x \cdot \nabla \psi(\sigma x) -\frac{\gamma_i (x-y) \cdot(x \psi(\sigma x)-y \psi(\sigma y)}{2|x-y|^2} \right)\, dxdy,
\end{align*}
where $i=1,2$. Taking a limit as $\sigma \to 0$ in \eqref{ph1}, we then conclude that
\begin{align} \label{ph2}
\begin{split}
\frac{N-2}{2} \int_{\R^N} |\nabla u|^2 \, dx + \frac N 2 \int_{\R^N} |u|^2 \, dx =&\left(\frac N 2 - \frac{\gamma_1}{4} \right) \int_{\R^N} \int_{\R^N} \frac{|u(x)|^2 |u(y)|^2}{|x-y|^{\gamma_1}} \, dxdy \\
& \quad - \left(\frac N 2-\frac{\gamma_2}{4} \right) \int_{\R^N} \int_{\R^N} \frac{|u(x)|^2 |u(y)|^2}{|x-y|^{\gamma_2}} \, dxdy.
\end{split}
\end{align}
On the other hand, multiplying \eqref{equ} by $u$ and integrating on $\R^N$, we get that
$$
\frac{N}{2} \int_{\R^N} |\nabla u|^2 \, dx + \frac N 2 \int_{\R^N} |u|^2 \, dx = \frac N 2\int_{\R^N} \int_{\R^N} \frac{|u(x)|^2 |u(y)|^2}{|x-y|^{\gamma_1}} \, dxdy - \frac N 2\int_{\R^N} \int_{\R^N} \frac{|u(x)|^2 |u(y)|^2}{|x-y|^{\gamma_2}} \, dxdy. \\
$$
This along with \eqref{ph2} gives rise to the result of the lemma. Thus the proof  is completed.
\end{proof}

\section{Mass subcritical case}\label{section3}

In this section, we consider the mass subcritical case $0<\gamma_2<\gamma_1<\min\{N, 2\}$. The aim is to establish Proposition \ref{prop}, Theorems \ref{thmsub1} and \ref{thmsub2}. To begin with, for any $u \in S(c)$ and $\theta>0$, we define
\begin{align} \label{scaling1}
u^{\theta}(x):=\theta^\frac{1+\beta N}{2}u(\theta^{\beta} x) \quad x \in \R^N.
\end{align}
It is clear to see that $u^{\theta} \in S(\theta c)$ and
\begin{align} \label{scaling2}
\begin{split}
E(u^{\theta})&=\frac{\theta^{1+2\beta}}{2}\int_{\R^N} |\nabla u|^2 \, dx + \frac {\theta^{2+ \beta \gamma_2}}{4}\int_{\R^N} \int_{\R^N}\frac{|u(x)|^2|u(y)|^2}{|x-y|^{\gamma_2}} \, dxdy\\
& \quad - \frac {\theta^{2+ \beta \gamma_1}}{4}\int_{\R^N} \int_{\R^N}\frac{|u(x)|^2|u(y)|^2}{|x-y|^{\gamma_1}} \, dxdy.
\end{split}
\end{align}


\begin{proof}[Proof of Proposition \ref{prop}]
Let us first prove the assertion $(\textnormal{i})$. In light of \eqref{inequ}, we have that
\begin{align} \label{belimit}
\begin{split}
E(u) &\geq \frac 12 \int_{\R^N}|\nabla u|^2 \,dx + \frac {1}{4}\int_{\R^N} \int_{\R^N}\frac{|u(x)|^2|u(y)|^2}{|x-y|^{\gamma_2}} \, dxdy \\
& \quad -C_{N, \gamma_1} \left(\int_{\R^N} |\nabla u|^2 \, dx\right)^{\frac{\gamma_1}{2}} \left(\int_{\R^N} |u|^2 \, dx\right)^{\frac{4-\gamma_1}{2}},
\end{split}
\end{align}
from which we can derive that $m(c)>-\infty$ for any $c>0$, because of $0<\gamma_1<2$. It yields from \eqref{scaling} that $E(u_t) \to 0$ as $t \to 0^+$. This then indicates that $m(c) \leq 0$. We next verify the assertion $(\textnormal{ii})$. For any $c>0$, let $\{c_n\} \subset \R^+$ be such that $c_n \to c$ as $n \to \infty$ and $\{u_n\} \subset S(c_n)$ be such that
\begin{align} \label{eun}
E(u_n) \leq m(c_n) + \frac 1 n.
\end{align}
Define 
$$
v_n:=\frac{u_n}{\|u_n\|_2} c^{\frac 12} \in S(c).
$$
Since $m(c_n)<0$, by \eqref{eun}, then $E(u_n) \leq C$ for some $C>0$. Thanks to $\gamma_1<2$, from \eqref{belimit}, we have that $\{u_n\}$ is bounded in $H^1(\R^N)$. As a result, we see that
\begin{align*}
m(c) \leq E(v_n) &=\frac{c}{2\|u_n\|_2^2} \int_{\R^N} |\nabla u_n|^2 \,dx +\frac{c^2}{4\|u_n\|_2^4}\int_{\R^N} \int_{\R^N}\frac{|u_n(x)|^2|u_n(y)|^2}{|x-y|^{\gamma_2}} \, dxdy \\
&\quad -\frac{c^2}{4\|u_n\|_2^4}\int_{\R^N} \int_{\R^N}\frac{|u_n(x)|^2|u_n(y)|^2}{|x-y|^{\gamma_1}} \, dxdy\\
&=E(u_n) +o_n(1)  \leq m(c_n) + o_n(1).
\end{align*}
On the other hand, from the definition of $m(c)$, we know that there exists $\widetilde{u}_n \in S(c)$ such that
\begin{align} \label{eun1}
E(\widetilde{u}_n) \leq m(c) + \frac 1 n.
\end{align}
Define 
$$
\widetilde{v}_n:=\frac{\widetilde{u}_n}{\|\widetilde{u}_n\|_2} c_n^{\frac 12} \in S(c_n).
$$
It follows from \eqref{eun1} that $E(\widetilde{u}_n) \leq C$ for some $C>0$. Due to $\gamma_1<2$, from \eqref{belimit}, we also have that $\{\widetilde{u}_n\}$ is bounded in $H^1(\R^N)$. Therefore, we obtain that
\begin{align*}
m(c_n) \leq E(\widetilde{v}_n) &=\frac{c_n}{2\|\widetilde{u}_n\|_2^2} \int_{\R^N} |\nabla \widetilde{u}_n|^2 \,dx +\frac{c^2_n}{4\|\widetilde{u}_n\|_2^4}\int_{\R^N} \int_{\R^N}\frac{|\widetilde{u}_n(x)|^2|\widetilde{u}_n(y)|^2}{|x-y|^{\gamma_2}} \, dxdy \\
&\quad -\frac{c_n^2}{4\|\widetilde{u}_n\|_2^4}\int_{\R^N} \int_{\R^N}\frac{|\widetilde{u}_n(x)|^2|\widetilde{u}_n(y)|^2}{|x-y|^{\gamma_1}} \, dxdy\\
&=E(\widetilde{u}_n) +o_n(1)  \leq m(c) + o_n(1).
\end{align*}
Therefore, we obtain that $m(c_n)=m(c) + o_n(1)$ and the assertion $(\textnormal{ii})$ follows. We now deduce the assertion $(\textnormal{iii})$. From the definition of $m(c)$, we know that, for any $\eps>0$, there exist $u_1 \in S(c_1)$ and $u_2 \in S(c_2)$ such that
$$
E(u_1) \leq m(c_1) + \frac{\eps}{4}, \quad E(u_2) \leq m(c_2) + \frac{\eps}{4}.
$$
Since $C^{\infty}_0(\R^N)$ is dense in $H^1(\R^N)$, then there exist $\widetilde{u}_{1,n}, \widetilde{u}_{2,n} \in C^{\infty}_0(\R^N)$ such that $\|\widetilde{u}_{1,n}-u_1\| =o_n(1)$ and $\|\widetilde{u}_{1,n}-u_2\| =o_n(1)$. In particular, there holds that $\|\widetilde{u}_{1,n}-u_1\|_2 =o_n(1)$ and $\|\widetilde{u}_{2,n}-u_2\|_2 =o_n(1)$. Define
$$
\widetilde{v}_{1,n}:=\frac{\widetilde{u}_{1,n}}{\|\widetilde{u}_{1,n}\|_2} c_1^{\frac 12}\in S(c_1), \quad \widetilde{v}_{2,n}:=\frac{\widetilde{u}_{2,n}}{\|\widetilde{u}_{1,n}\|_2} c_2^{\frac 12} \in S(c_2).
$$
It then follows that $\|\widetilde{v}_{1,n}-u_1\| =o_n(1)$ and $\|\widetilde{v}_{1,n}-u_2\| =o_n(1)$. This necessarily leads to
$$
E(\widetilde{v}_{1,n}) \leq m(c_1) + \frac{\eps}{4} +o_n(1), \quad E(\widetilde{v}_{2,n}) \leq m(c_2) + \frac{\eps}{4}+o_n(1).
$$
Therefore, we are able to derive that there exist $v_1 \in C^{\infty}_0(\R^N) \cap S(c_1)$ and $v_2 \in C^{\infty}_0(\R^N) \cap S(c_2)$ such that
$$
E(v_1) \leq m(c_1) + \frac{\eps}{2}, \quad E(v_2) \leq m(c_2) + \frac{\eps}{2}.
$$
Since $E$ is invariant under any translation in $\R^N$, then, for any $h_1, h_2 \in \R^N$,
$$
E(v_1(\cdot-h_1)) \leq m(c_1) + \frac{\eps}{2}, \quad E(v_2(\cdot-h_2)) \leq m(c_2) + \frac{\eps}{2}.
$$
Consequently, without restriction, we may assume that, for any $\eps>0$, there exist $u_1 \in C^{\infty}_0(\R^N) \cap S(c_1)$ and $u_2 \in C^{\infty}_0(\R^N) \cap S(c_2)$ satisfying $\textnormal{dist}(\textnormal{supp} \,u_1, \, \textnormal{supp} \,u_2)>n$ such that
$$
E(u_1) \leq m(c_1) + \frac{\eps}{2}, \quad E(u_2) \leq m(c_2) + \frac{\eps}{2}.
$$ 
Then $u=u_1+u_2 \in S(c)$ and
$$
\int_{\R^N} |\nabla u|^2 \, dx =\int_{\R^N} |\nabla u_1|^2 \, dx + \int_{\R^N} |\nabla u_2|^2 \,dx.
$$
Observe that
\begin{align*}
|u_1(x)+u_2(x)|^2|u_1(y)+u_2(y)|^2&=|u_1(x)|^2|u_1(y)|^2+|u_1(x)|^2|u_2(y)|^2+2|u_1(x)|^2 \left(u_1(y)u_2(y)\right)\\
& \quad + |u_2(x)|^2|u_1(y)|^2+|u_2(x)|^2|u_2(y)|^2+2|u_2(x)|^2\left(u_1(y)u_2(y)\right)\\
& \quad +2|u_1(y)|^2\left(u_1(x)u_2(x)\right) +2|u_2(y)|^2 \left(u_1(x)u_2(x)\right) \\
& \quad + 4 \left(u_1(x)u_2(x)u_1(y)u_2(y)\right),
\end{align*}
In addition, we see that
\begin{align*}
\int_{\R^N} \int_{\R^N}\frac{|u_1(x)|^2|u_2(y)|^2}{|x-y|^{\gamma_i}} \, dxdy &=\int_{\textnormal{supp} \, u_1} \int_{\textnormal{supp} \, u_2} \frac{|u_1(x)|^2|u_2(y)|^2}{|x-y|^{\gamma_i}} \, dxdy \leq \frac{c_1 c_2}{n^{\gamma_i}},
\end{align*}
\begin{align*}
\int_{\R^N} \int_{\R^N}\frac{|u_2(x)|^2 |u_1(y)|^2}{|x-y|^{\gamma_i}}\, dxdy=\int_{\textnormal{supp} \, u_2} \int_{\textnormal{supp} \, u_1} \frac{|u_2(x)|^2|u_1(y)|^2}{|x-y|^{\gamma_i}} \, dxdy\leq \frac{c_1 c_2}{n^{\gamma_i}},
\end{align*}
and
\begin{align*}
\int_{\R^N} \int_{\R^N}\frac{|u_1(x)|^2|u_2(y)||u_1(y)|}{|x-y|^{\gamma_i}} \, dxdy=\int_{\textnormal{supp} \, u_1} \int_{\textnormal{supp} \, u_2}\frac{|u_1(x)|^2|u_1(y)||u_2(y)|}{|x-y|^{\gamma_i}} \, dxdy \leq \frac{c_1(c_1+c_2)}{2n^{\gamma_i}},
\end{align*}
\begin{align*}
\int_{\R^N} \int_{\R^N}\frac{|u_1(x)||u_2(x)||u_1(y)|^2}{|x-y|^{\gamma_i}} \, dxdy &=\int_{\textnormal{supp} \, u_2} \int_{\textnormal{supp} \, u_1}\frac{|u_1(x)||u_2(x)||u_1(y)|^2}{|x-y|^{\gamma_i}} \, dxdy \leq \frac{c_1(c_1+c_2)}{2n^{\gamma_i}}.
\end{align*}
Similarly, we can derive that
$$
\int_{\R^N} \int_{\R^N}\frac{|u_2(x)|^2|u_1(y)||u_2(y)|}{|x-y|^{\gamma_i}} \, dxdy \leq \frac{c_2(c_1+c_2)}{2n^{\gamma_i}},
$$
$$
\int_{\R^N} \int_{\R^N}\frac{|u_2(x)||u_1(x)||u_2(y)|^2}{|x-y|^{\gamma_i}} \, dxdy \leq \frac{c_2(c_1+c_2)}{2n^{\gamma_i}}
$$
and
$$
\int_{\R^N} \int_{\R^N}\frac{|u_1(x)||u_2(x)||u_1(y)||u_2(y)|}{|x-y|^{\gamma_i}} \, dxdy  \leq \frac{(c_1+c_2)^2}{4n^{\gamma_i}}.
$$
Accordingly, we arrive at
\begin{align*}
\int_{\R^N} \int_{\R^N}\frac{|u(x)|^2|u(y)|^2}{|x-y|^{\gamma_i}} \, dxdy &= \int_{\R^N} \int_{\R^N}\frac{|u_1(x)|^2|u_1(y)|^2}{|x-y|^{\gamma_i}} \, dxdy + \int_{\R^N} \int_{\R^N}\frac{|u_2(x)|^2|u_2(y)|^2}{|x-y|^{\gamma_i}} \, dxdy +o_n(1),
\end{align*}
where $i=1,2$. Therefore, we get that
\begin{align*}
m(c) \leq E(u) = E(u_1) + E(u_2)  +o_n(1) \leq m(c_1) +m(c_2) +\eps+o_n(1).
\end{align*}
This then leads to $m(c) \leq m(c_1) + m(c_2)$ for $c=c_1+c_2$. We finally demonstrate the assertion $(\textnormal{iv})$. For any $u_c \in S(c)$ with $E(u_c) \leq 0$, from \eqref{belimit}, we get that
$$
C_{N, \gamma_1} c^{\frac{4-\gamma_1}{2}}  \left(\int_{\R^N} |\nabla u_c|^2 \, dx\right)^{\frac{\gamma_1}{2}} \geq \frac 12 \int_{\R^N}|\nabla u_c|^2 \,dx + \frac {1}{4}\int_{\R^N} \int_{\R^N}\frac{|u_c(x)|^2|u_c(y)|^2}{|x-y|^{\gamma_2}} \, dxdy \geq \frac 12 \int_{\R^N}|\nabla u_c|^2 \,dx.
$$
Since $\gamma_1<2$, then 
$$
\int_{\R^N} |\nabla u_c|^2 \, dx \leq \left(2C_{N, \gamma_1} c^{\frac{4-\gamma_1}{2}} \right)^{\frac{2}{2-\gamma_1}}\to 0^+ \quad \mbox{as} \,\, c \to 0^+.
$$
Making use of \eqref{inequ}, we then get that $E(u_c) \to 0$ as $c \to 0^+$. It follows that $m(c) \to 0$ as $c \to 0^+$. Let $u \in S(1)$ and $\beta:=\frac{1}{2-\gamma_2}>0$, by \eqref{scaling2}, we derive that
$$
E(u^{\theta})=\frac{\theta^{1+2\beta}}{2}\int_{\R^N} |\nabla u|^2 \, dx + \frac {\theta^{1+ 2\beta}}{4}\int_{\R^N} \int_{\R^N}\frac{|u(x)|^2|u(y)|^2}{|x-y|^{\gamma_2}} \, dxdy- \frac {\theta^{2+\beta \gamma_1}}{4}\int_{\R^N} \int_{\R^N}\frac{|u(x)|^2|u(y)|^2}{|x-y|^{\gamma_1}} \, dxdy.
$$
Thanks to $\gamma_2<\gamma_1$, then $E(u^{\theta}) \to -\infty$ as $\theta \to \infty$. Note that $u^{\theta} \in S(\theta)$, then $m(c) \to -\infty$ as $c \to \infty$. Thus the proof is completed.
\end{proof}

\begin{lem} \label{estimate}
Let $0<\gamma_2<\gamma_1<\min\{N, 2\}$. Then there exists a constant $C>0$ depending only on $\gamma_1, \gamma_2$ and $N$ such that
\begin{align} \label{inequ3}
\begin{split}
\hspace{-1cm}\int_{\R^N}\int_{\R^N}\frac{|u(x)|^2|u(y)|^2}{|x-y|^{\gamma_1}} \, dxdy &\leq C\left(\int_{\R^N}\int_{\R^N}\frac{|u(x)|^2|u(y)|^2}{|x-y|^{\gamma_2}} \, dxdy\right)^{\theta} \|\nabla u\|_2^{\gamma(1-\theta)}\|u\|_2^{(4-\gamma)(1-\theta)},
\end{split}
\end{align}
where $\gamma>0$ is a constant satisfying $0<\gamma_2<\gamma_1<\gamma<\min\{N, 2\}$ and $\gamma_1=\theta \gamma_2+(1-\theta) \gamma$ for $0<\theta<1$. If $N \geq 3$, then \eqref{inequ3} holds true for $\gamma=2$.
\end{lem}
\begin{proof}
In light of H\"older's inequality, \eqref{HLS} and \eqref{GN}, we are able to derive that
\begin{align*}
\int_{\R^N}\int_{\R^N}\frac{|u(x)|^2|u(y)|^2}{|x-y|^{\gamma_1}} \, dxdy &\leq \left(\int_{\R^N}\int_{\R^N}\frac{|u(x)|^2|u(y)|^2}{|x-y|^{\gamma_2}}\, dxdy\right)^{\theta}\left(\int_{\R^N}\int_{\R^N}\frac{|u(x)|^2|u(y)|^2}{|x-y|^{\gamma}} \, dxdy\right)^{1-\theta} \\
& \leq C\left(\int_{\R^N}\int_{\R^N}\frac{|u(x)|^2|u(y)|^2}{|x-y|^{\gamma_2}}\, dxdy\right)^{\theta}
\|u\|_{2p}^{2(1-\theta)}\|u\|_{2r}^{2(1-\theta)} \\
& \leq C\left(\int_{\R^N}\frac{|u(x)|^2|u(y)|^2}{|x-y|^{\gamma_2}}\, dxdy\right)^{\theta} \|\nabla u\|_2^{\gamma(1-\theta)} \|u\|_2^{(4-\gamma)(1-\theta)}.
\end{align*}
If $N \geq 3$, from estimate above, we can take $\gamma=2$ to deduce that \eqref{inequ3} remains valid. Thus we have completed the proof.
\end{proof}

\begin{proof} [Proof of Theorem \ref{thmsub1}]
From Lemma \ref{estimate} with $N \geq 3$, we have that, for any $u \in S(c)$,
\begin{align*}
E(u) & \geq \frac 12 \int_{\R^N}|\nabla u|^2 \,dx + \frac {1}{4}\int_{\R^N} \int_{\R^N}\frac{|u(x)|^2|u(y)|^2}{|x-y|^{\gamma_2}} \, dxdy\\
& \quad - C_{N, \gamma_1, \gamma_2} \left(\int_{\R^N}\frac{|u(x)|^2|u(y)|^2}{|x-y|^{\gamma_2}} \, dxdy\right)^{\theta} \left(\int_{\R^N} |\nabla u|^2 \, dx\right)^{1-\theta} c^{1-\theta},
\end{align*}
where $0<\theta<1$ and $\gamma_1=\theta \gamma_2 + 2 (1-\theta)$. Using Young's inequality, we then get that
\begin{align*}
E(u) & \geq \left(\frac 12 -\theta C_{N, \gamma_1, \gamma_2}c^{1-\theta} \right) \int_{\R^N}|\nabla u|^2 \,dx + \left(\frac 14 -(1-\theta) C_{N, \gamma_1, \gamma_2}c^{1-\theta} \right) \int_{\R^N} \int_{\R^N}\frac{|u(x)|^2|u(y)|^2}{|x-y|^{\gamma_2}} \, dxdy.
\end{align*}
Therefore, there exists a constant $\hat{c}_0>0$ such that, for any $0<c< \hat{c}_0$,
\begin{align} \label{nonexistence}
E(u) \geq \frac 3 8 \int_{\R^N}|\nabla u|^2 \,dx + \frac 1 8 \int_{\R^N} \int_{\R^N}\frac{|u(x)|^2|u(y)|^2}{|x-y|^{\gamma_2}} \, dxdy > 0.
\end{align} 
As a consequence, we then get that $m(c) \geq 0$ for any $0<c<\hat{c_0}$. In view of the assertion $(\textnormal{i})$ of Proposition \ref{prop}, we then obtain that $m(c)=0$ for any $0<c<\hat{c}_0$. We now prove that, for any $0<c<\hat{c}_0$, there exists no minimizers to \eqref{gmin}. If there were a minimizer $u \in S(c)$ to \eqref{gmin} for some $0<c<\hat{c}_0$, then $E(u)=0$. This clearly contradicts \eqref{nonexistence}. Thus we have completed the proof.
\end{proof}

\begin{proof}[Proof of Theorem \ref{thmsub2}]
For any $u \in S(c)$ and $t >0$, by applying \eqref{scaling}, we first see that
\begin{align*}
E(u_t)&=\frac {t^{\gamma_2}}{4} \left( 2 t^{2-\gamma_2}\int_{\R^N}|\nabla u|^2 \,dx + \int_{\R^N} \int_{\R^N}\frac{|u(x)|^2|u(y)|^2}{|x-y|^{\gamma_2}} \, dxdy-t^{\gamma_1-\gamma_2} \int_{\R^N} \int_{\R^N}\frac{|u(x)|^2|u(y)|^2}{|x-y|^{\gamma_1}} \, dxdy\right) \\
&=:\frac {t^{\gamma_2}}{4} F_u(t),
\end{align*}
where
$$
F_u(t):=2 t^{2-\gamma_2}\int_{\R^N}|\nabla u|^2 \,dx + \int_{\R^N} \int_{\R^N}\frac{|u(x)|^2|u(y)|^2}{|x-y|^{\gamma_2}} \, dxdy-t^{\gamma_1-\gamma_2} \int_{\R^N} \int_{\R^N}\frac{|u(x)|^2|u(y)|^2}{|x-y|^{\gamma_1}} \, dxdy.
$$
Direct calculations yield that $F_u$ has a unique minimizer point $t_0>0$ defined by
$$
t_0:=\left(\frac{(\gamma_1-\gamma_2)\int_{\R^N} \int_{\R^N}\frac{|u(x)|^2|u(y)|^2}{|x-y|^{\gamma_1}} \, dxdy}{2(2-\gamma_2)\int_{\R^N}|\nabla u|^2 \,dx}\right)^{\frac{1}{2-\gamma_1}}
$$
and
\begin{align*}
F_u(t_0)&=\int_{\R^N} \int_{\R^N}\frac{|u(x)|^2|u(y)|^2}{|x-y|^{\gamma_2}} \, dxdy-C_{\gamma_1, \gamma_2}\frac{\left(\int_{\R^N} \int_{\R^N}\frac{|u(x)|^2|u(y)|^2}{|x-y|^{\gamma_1}} \, dxdy\right)^{\frac{2-\gamma_2}{2-\gamma_1}}}{\left(\int_{\R^N}|\nabla u|^2 \,dx\right)^{\frac{\gamma_1-\gamma_2}{2-\gamma_1}}},
\end{align*}
where
$$
C_{\gamma_1, \gamma_2}:=\frac{2(2-\gamma_1)}{\gamma_1-\gamma_2} \left(\frac{\gamma_1-\gamma_2}{2(2-\gamma_2)}\right)^{\frac{2-\gamma_2}{\gamma_1-\gamma_2}}.
$$
Let us now choose $u \in S(c)$ as
$$
u=\frac{Q_{\gamma_1}}{\|Q_{\gamma_1}\|_2}c^{\frac 12} \in S(c),
$$
where $Q_{\gamma_1} \in H^1(\R^N)$ is such that the equality in \eqref{inequ} holds for $\gamma=\gamma_1$. By using \eqref{inequ}, we then have that
\begin{align*}
F_u(t_0)&=\frac{c^2}{\|Q_{\gamma_1}\|_2^4}\int_{\R^N} \int_{\R^N}\frac{|Q_{\gamma_1}(x)|^2|Q_{\gamma_1}(y)|^2}{|x-y|^{\gamma_2}} \, dxdy -C_{\gamma_1, \gamma_2}C_{N, \gamma_1}^{\frac {2-\gamma_2}{2-\gamma_1}} \frac{c^{1+\frac{2-\gamma_2}{2-\gamma_1}}}{\|Q_{\gamma_1}\|_2^{\gamma_2}}\left(\int_{\R^N} |\nabla Q_{\gamma_1}|^2 \, dx\right)^{\frac{\gamma_1}{2}} \\
& \leq C_{N, \gamma_2}\frac{c^2}{\|Q_{\gamma_1}\|_2^{\gamma_2}}\left(\int_{\R^N} |\nabla Q_{\gamma_1}|^2 \, dx\right)^{\frac{\gamma_2}{2}}-C_{\gamma_1, \gamma_2}C_{N, \gamma_1}^{\frac {2-\gamma_2}{2-\gamma_1}} \frac{c^{1+\frac{2-\gamma_2}{2-\gamma_1}}}{\|Q_{\gamma_1}\|_2^{\gamma_2}}\left(\int_{\R^N} |\nabla Q_{\gamma_1}|^2 \, dx\right)^{\frac{\gamma_1}{2}}\\
& = \frac{c^2}{\|Q_{\gamma_1}\|_2^{\gamma_2}} \left(\int_{\R^N} |\nabla Q_{\gamma_1}|^2 \, dx\right)^{\frac{\gamma_2}{2}} \left(C_{N, \gamma_2}-c^{\frac{\gamma_1-\gamma_2}{2-\gamma_1}}C_{\gamma_1, \gamma_2}C_{N, \gamma_1}^{\frac {2-\gamma_2}{2-\gamma_1}} \right).
\end{align*}
This infers that there exists a constant $\tilde{c}_0>0$ depending only on $\gamma_1, \gamma_1$ and $N$ such that, for any $c>\tilde{c}_0$, $F_u(t_0)<0$. As a consequence, we get that $m(c)<0$ for any $c>\tilde{c}_0$. 

Next we prove that \eqref{ssub} holds true.  To do this, it suffices to assert that $m(\theta c) <\theta m(c)$ for any $\theta>1$. Since $m(c)<0$ for any $c>\tilde{c}_0$, then, for any $\eps>0$ small enough, there exists $u \in S(c)$ such that $E(u) \leq m(c) + \eps <0$. Let $\beta:=\frac{1}{2-\gamma_1}>0$, it then follows from \eqref{scaling2} that
$$
E(u^{\theta})=\frac{\theta^{1+2\beta}}{2}\int_{\R^N} |\nabla u|^2 \, dx + \frac {\theta^{2+ \beta \gamma_2}}{4}\int_{\R^N} \int_{\R^N}\frac{|u(x)|^2|u(y)|^2}{|x-y|^{\gamma_2}} \, dxdy- \frac {\theta^{1+ 2\beta}}{4}\int_{\R^N} \int_{\R^N}\frac{|u(x)|^2|u(y)|^2}{|x-y|^{\gamma_1}} \, dxdy.
$$
Define $f_u(\theta):=E(u^{\theta}) - \theta E(u)$ for any $\theta>1$, then
\begin{align*}
f_u(\theta)&= \frac{\theta^{1+2\beta}-\theta}{2}\int_{\R^N} |\nabla u|^2 \, dx+\frac {\theta^{2+ \beta \gamma_2}-\theta}{4}\int_{\R^N} \int_{\R^N}\frac{|u(x)|^2|u(y)|^2}{|x-y|^{\gamma_2}} \, dxdy \\
& \quad -\frac {\theta^{1+ 2\beta}-\theta}{4}\int_{\R^N} \int_{\R^N}\frac{|u(x)|^2|u(y)|^2}{|x-y|^{\gamma_1}} \, dxdy.
\end{align*}
By simple calculations, we obtain that
\begin{align*}
\frac{d}{d \theta} f_u(\theta)&=\left(\left(1+2 \beta\right) \theta^{2\beta} -1\right) \left(\frac 12 \int_{\R^N} |\nabla u|^2 \, dx - \frac 14 \int_{\R^N} \int_{\R^N}\frac{|u(x)|^2|u(y)|^2}{|x-y|^{\gamma_1}} \, dxdy\right) \\
& \quad +\frac {\left(2+ \beta \gamma_2\right) \theta^{1+\beta \gamma_2}-1}{4}\int_{\R^N} \int_{\R^N}\frac{|u(x)|^2|u(y)|^2}{|x-y|^{\gamma_2}} \, dxdy.
\end{align*}
Since $E(u)<0$, $\gamma_1>\gamma_2$ and $\theta>1$, then $\frac{d}{d \theta} f_u(\theta)<0$. Note that $u^{\theta} \in S(\theta c)$, then $m(\theta c) \leq E(u_\theta) < \theta E(u) \leq \theta m(c) + \theta \eps$. This readily infers that $m(\theta c) < \theta m(c)$ for any $\theta >1$. Therefore, we derive that \eqref{ssub} holds true. Utilizing the Lions concentration compactness principle, see \cite{Li1, Li2}, we then get the desired result. Thus the proof is completed.
\end{proof}

\section{Mass supercritical case} \label{section4}

In this section, we study the mass supercritical case $0<\gamma_2<\gamma_1<\min\{N, 4\}$, $\gamma_1>2$ and $N \geq 3$ and Theorem \ref{thm1} is proved. 

\begin{lem} \label{unique}
Let $0<\gamma_2<\gamma_1<\min\{N, 4\}$, $\gamma_1>2$ and $N \geq 3$. Then for any $u \in S(c)$, there exists a unique $t_u >0$ such that $u_{t_u} \in P(c)$ and $E(u_{t_u})=\max_{t>0}E(u_t)$. In addition, $t_u \leq 1$ if $Q(u) \leq 0$ and the function $t \mapsto E(u_t)$ is concave on $[t_u, +\infty)$. 
\end{lem}
\begin{proof}
In virtue of \eqref{scaling}, we know that, for any $u \in S(c)$,
\begin{align*}
\frac{\partial}{\partial t}E(u_t)&= t \int_{\R^N}|\nabla u|^2 \,dx + \frac {\gamma_2t^{\gamma_2-1}}{4}\int_{\R^N} \int_{\R^N}\frac{|u(x)|^2|u(y)|^2}{|x-y|^{\gamma_2}} \, dxdy \\
&\quad -\frac {\gamma_1t^{\gamma_1-1}}{4} \int_{\R^N} \int_{\R^N}\frac{|u(x)|^2|u(y)|^2}{|x-y|^{\gamma_1}} \, dxdy=\frac{Q(u_t)}{t}.
\end{align*}
Thanks to $\gamma_1>2$, then $\frac{\partial}{\partial t}E(u_t)>0$ for any $t>0$ small enough and $\frac{\partial}{\partial t}E(u_t) \to -\infty$ as $t \to \infty$. If $0<\gamma_2 \leq 2$, we write
\begin{align*}
\frac{\partial}{\partial t}E(u_t)&= t^{\gamma_2-1} \left(t^{2-\gamma_2}\int_{\R^N}|\nabla u|^2 \,dx + \frac {\gamma_2}{4}\int_{\R^N} \int_{\R^N}\frac{|u(x)|^2|u(y)|^2}{|x-y|^{\gamma_2}} \, dxdy \right. \\
&\qquad \qquad \quad \left.-\frac {\gamma_1t^{\gamma_1-\gamma_2}}{4} \int_{\R^N} \int_{\R^N}\frac{|u(x)|^2|u(y)|^2}{|x-y|^{\gamma_1}} \, dxdy \right).
\end{align*}
This clearly gives that there exists a unique $t_u >0$ such that $\frac{\partial}{\partial t}E(u_t)_{\mid_{t=t_u}}=0$. Furthermore, we see that $\frac{\partial}{\partial t}E(u_t)>0$ for any $0<t<t_u$ and  $\frac{\partial}{\partial t}E(u_t)<0$ for any $t>t_u$. Therefore, we have that $E(u_{t_u})=\max_{t>0}E(u_t)$ and $t_u \leq 1$ if $Q(u) \leq 0$. 
Observe that
\begin{align*}
\frac{\partial^2}{\partial t^2}E(u_t)&=\int_{\R^N}|\nabla u|^2 \,dx + \frac {\gamma_2(\gamma_2-1)t^{\gamma_2-2}}{4}\int_{\R^N} \int_{\R^N}\frac{|u(x)|^2|u(y)|^2}{|x-y|^{\gamma_2}} \, dxdy \\
&\quad - \frac {\gamma_1(\gamma_1-1)t^{\gamma_1-2}}{4}\int_{\R^N} \int_{\R^N}\frac{|u(x)|^2|u(y)|^2}{|x-y|^{\gamma_1}} \, dxdy \\
&=t^{\gamma_2-2} \left(t^{2-\gamma_2}\int_{\R^N}|\nabla u|^2 \,dx + \frac {\gamma_2(\gamma_2-1)}{4}\int_{\R^N} \int_{\R^N}\frac{|u(x)|^2|u(y)|^2}{|x-y|^{\gamma_2}} \, dxdy \right. \\
&\qquad \qquad \quad \left.-\frac {\gamma_1(\gamma_1-1)t^{\gamma_1-\gamma_2}}{4} \int_{\R^N} \int_{\R^N}\frac{|u(x)|^2|u(y)|^2}{|x-y|^{\gamma_1}} \, dxdy \right).
\end{align*}
Since $0<\gamma_2 \leq 2$, $\gamma_1>2$ and $\frac{\partial}{\partial t}E(u_t)<0$ for any $t>t_u$, then $\frac{\partial^2}{\partial t^2}E(u_t)<0$ for any $t>t_u$. This means that the function $t \mapsto E(u_t)$ is concave on $[t_u, +\infty)$. By a similar way, we can also verify the same results for the case $\gamma_2>2$. Thus the proof is completed.
\end{proof}

\begin{lem} \label{below}
Let $0<\gamma_2<\gamma_1<\min\{N, 4\}$, $\gamma_1>2$ and $N \geq 3$,  then $\Gamma(c)>0$ and $E$ restricted on $P(c)$ is coercive for any $c>0$.
\end{lem}
\begin{proof}
We first prove that $\Gamma(c)>0$ for any $c>0$. For any $u \in P(c)$, we find that
\begin{align}\label{b1}
\hspace{-1cm}E(u)=E(u)-\frac{1}{\gamma_1}Q(u)=\frac{\gamma_1-2}{2\gamma_1} \int_{\R^N}|\nabla u|^2 \, dx + \frac{\gamma_1-\gamma_2}{4\gamma_1}  \int_{\R^N} \int_{\R^N}\frac{|u(x)|^2|u(y)|^2}{|x-y|^{\gamma_2}} \, dxdy.
\end{align}
Moreover, by applying \eqref{inequ}, we know that
\begin{align} \label{largenable1}
\begin{split}
\int_{\R^N}|\nabla u|^2 \,dx + \frac {\gamma_2}{4}\int_{\R^N} \int_{\R^N}\frac{|u(x)|^2|u(y)|^2}{|x-y|^{\gamma_2}} \, dxdy&=\frac{\gamma_1}{4} \int_{\R^N} \int_{\R^N}\frac{|u(x)|^2|u(y)|^2}{|x-y|^{\gamma_1}} \, dxdy \\
& \leq \frac{\gamma_1 C_{N,\gamma_1}}{4} \left(\int_{\R^N} |\nabla u|^2 \, dx \right)^{\frac{\gamma_1}{2}} c^{\frac{4-\gamma_1}{2}}.
\end{split}
\end{align}
This implies that
\begin{align} \label{largenable}
\int_{\R^N}|\nabla u_c|^2 \,dx \geq \left(\frac{4}{\gamma_1C_{N,\gamma_1}c^{\frac{4-\gamma_1}{2}}}\right)^{\frac{2}{\gamma_1-2}}.
\end{align}
By means of \eqref{b1}, we then have the desired conclusion. According to \eqref{b1}, we see that, for any sequence $\{u_n\} \subset P(c)$ with $\|u_n\| \to \infty$ as $n \to \infty$, there holds that $E(u_n) \to \infty$ as $n \to \infty$. This indicates that $E$ restricted on $P(c)$ is coercive for any $c>0$. Thus we have completed the proof.
\end{proof}

\begin{defi}\label{homotopy} \cite[Definition 3.1]{Gh}
Let $B$ be a closed subset of a set $Y \subset H^1(\R^N)$. We say that a class $\mathcal{G}$ of compact subsets of $Y$ is a homotopy stable family with the closed boundary $B$ provided that
\begin{enumerate}
\item [\textnormal{(i)}] every set in $\mathcal{G}$ contains $B$;
\item [\textnormal{(ii)}] for any $A \in \mathcal{G}$ and any function $\eta \in C([0, 1] \times Y, Y)$ satisfying $\eta(t, x)=x$ for all $(t, x) \in (\{0\} \times Y) \cup([0, 1] \times B)$, then $\eta(\{1\} \times A) \in \mathcal{G}$.
\end{enumerate}
\end{defi}

\begin{lem}\label{ps}
Let $0<\gamma_2<\gamma_1<\min\{N, 4\}$, $\gamma_1>2$ and $N \geq 3$. Let $\mathcal{G}$ be a homotopy stable family of compact subsets of $S(c)$ with closed boundary $B$ and set
\begin{align} \label{ming}
\Gamma_{\mathcal{G}}(c):=\inf_{A\in \mathcal{G}}\max_{u \in A} F(u),
\end{align}
where $F(u):=E(u_{t_u})=\max_{t>0}E(u_t)$. Suppose that $B$ is contained on a connected component of $P(c)$ and $\max\{\sup F(B), 0\}<\Gamma_{\mathcal{G}}(c)<\infty$. Then there exists a Palais-Smale sequence $\{u_n\} \subset P(c)$ for $E$ restricted on $S(c)$ at the level $\Gamma_{\mathcal{G}}(c)$.
\end{lem}
\begin{proof}
To prove this result, we are inspired by \cite{BS1, BS2}. Let us first define a mapping $\eta: [0,1] \times S(c) \to S(c)$ by $\eta(s, u)=u_{1-s+st_u}$. In view of Lemma \ref{unique}, we know that $t_u=1$ if $u \in P(c)$.  Since $B \subset P(c)$, we then have that $\eta(s ,u)=u$ for any $(s, u) \in (\{0\} \times S(c)) \cup([0, 1] \times B)$. In addition, it is not hard to find that $\eta$ is continuous on $[0,1] \times S(c)$. Suppose now that $\{D_n\} \subset \mathcal{G}$ is a minimizing sequence to \eqref{ming}. According to Definition \ref{homotopy}, we then get that
$$
A_n:=\eta(\{1\} \times D_n)=\{u_{t_u} : u \in D_n\} \in \mathcal{G}.
$$
Notice that $A_n \subset P(c)$, then 
$$
\max_{v \in A_n}F(v)=\max_{u \in D_n}F(u).
$$ 
This suggests that there exists another minimizing sequence $\{A_n\} \subset P(c)$ to \eqref{ming}. By appplying \cite[Theorem 3.2]{Gh}, we then deduce that there exists a Palais-Smale sequence $\{\tilde{u}_n\} \subset S(c)$ for $F$ at the level $\Gamma_{\mathcal{G}}(c)$ such that 
\begin{align} \label{dist}
\mbox{dist}_{H^1}(\tilde{u}_n, A_n)=o_n(1).
\end{align}
For simplicity, we write $t_n=t_{\tilde{u}_n}$ and $u_n=(\tilde{u}_n)_{t_n}$. We now claim that there exists a constant $C>0$ such that $1/C \leq t_n \leq C$. Indeed, observe first that
$$
t_n^2=\frac{\int_{\R^N} |\nabla u_n|^2 \, dx}{\int_{\R^N} |\nabla \tilde{u}_n|^2 \, dx}.
$$
Since $E(u_n)=F(\tilde{u}_n)=\Gamma_{\mathcal{G}}(c)+o_n(1)$ and $\{u_n\} \subset P(c)$, it then follows from Lemma \ref{below} that there exists a constant $C_1>0$ such that $1/C_1 \leq \|u_n\|\leq C_1$. On the other hand, since $\{A_n\} \subset P(c)$ is a minimizing sequence to \eqref{ming}, from Lemma \ref{below}, we then find that $\{A_n\}$ is bounded in $H^1(\R^N)$. Thanks to \eqref{dist}, then $\{\tilde{u}_n\}$ is bounded in $H^1(\R^N)$, i.e. there exists a constant $C_2>0$ such that $\|\tilde{u}_n\| \leq C_2$. In addition, note that $A_n$ is compact for any $n\in \N$, then there exists $v_n \in A_n$ such that 
$$
\mbox{dist}_{H^1}(\tilde{u}_n, A_n)=\|\tilde{u}_n-v_n\|=o_n(1).
$$
Utilizing the assumption $\Gamma_{\mathcal{G}}(c)>0$ for any $c>0$, we get that $\|v_n\| \geq 1/C_2$. This then gives that
$$
\|\tilde{u}_n\| \geq \|v_n\|-\|\tilde{u}_n-v_n\|\geq \frac{1}{C_2}+o_n(1).
$$
Therefore, the claim follows.

We next show that $\{u_n\} \subset P(c)$ is a Palais-Smale sequence for $E$ restricted on $S(c)$ at the level $\Gamma_{\mathcal{G}}(c)$. We now denote by $\|\cdot\|_{*}$ the dual norm of $(T_u S(c))^*$. Accordingly, it holds that
\begin{align*}
\|dE(u_n)\|_*=\sup_{\psi \in T_{u_n}S(c), \|\psi\|\leq 1}|dE(u_n)[\psi]|=\sup_{\psi \in T_{u_n}S(c), \|\psi\|\leq 1} |dE(u_n)[(\psi_{\frac{1}{t_n}})_{t_n}]|.
\end{align*}
From straightforward calculations, we are able to derive that the mapping $T_uS(c) \to T_{u_{t_u}}S(c)$ defined by $\psi \mapsto \psi_{t_u}$ is an isomorphism. Moreover, we can check that $dF(u)[\psi]=dE(u_{t_u})[\psi_{t_u}]$ for any $u \in S(c)$ and $\psi \in T_uS(c)$. As a consequence, we have that
$$
\|dE(u_n)\|_*=\sup_{\psi \in T_{u_n}S(c), \|\psi\| \leq 1}|dF(\tilde{u}_n)[\psi_{\frac{1}{t_n}}]|.
$$
Since $\{\tilde{u}_n\} \subset S(c)$ is a Palais-Smale sequence for $F$ at the level $\Gamma_{\mathcal{G}}(c)$, we then  apply the claim to deduce that $\{u_n\} \subset P(c)$ is a Palais-Smale sequence for $E$ restricted on $S(c)$ at the level $\Gamma_{\mathcal{G}}(c)$. Thus the proof is completed.
\end{proof}

\begin{lem} \label{pss}
Let $0<\gamma_2<\gamma_1<\min\{N, 4\}$, $\gamma_1>2$ and $N \geq 3$. Then there exists a Palais-Smale sequence $\{u_n\} \subset P(c)$ for $E$ restricted on $S(c)$ at the level $\Gamma(c)$.
\end{lem}
\begin{proof}
Let $B=\emptyset$ and let $\mathcal{G}$ be all singletons in $S(c)$, i.e. $\mathcal{G}=\{\{u\}: u \in S(c)\}$. Defining a mapping $\eta : [0, 1] \times S(c) \to S(c)$ by $\eta(s ,u)=u_{1-s+st_u}$, from Definition \ref{homotopy}, we then conclude that $\mathcal{G}$ is a homotopy stable family without boundary. In light of \eqref{ming}, we then have that
$$
\Gamma_{\mathcal{G}}(c)=\inf_{u \in S(c)} \sup_{t>0} E(u_t).
$$
We are now going to prove that $\Gamma_{\mathcal{G}}(c)=\Gamma(c)$. From Lemma \ref{unique}, we know that, for any $u \in S(c)$, there exists a unique $t_u>0$ such that $u_{t_u} \in P(c)$ and $E(u_{t_u})=\max_{t >0}E(u_t)$. This then implies that
$$
\inf_{u \in S(c)} \sup_{t>0} E(u_t)  \geq \inf_{u \in P(c)} E(u).
$$
On the other hand, for any $u \in P(c)$, by Lemma \ref{unique}, we have that $E(u)=\max_{t >0}E(u_t)$. This then gives that
$$
\inf_{u \in S(c)} \sup_{t>0} E(u_t)  \leq \inf_{u \in P(c)} E(u).
$$
Therefore, we have that $\Gamma_{\mathcal{G}}(c)=\Gamma(c)$. It then follows from Lemma \ref{ps} that the result of this lemma holds true and the proof is completed.
\end{proof}

\begin{lem} \label{sign}
Let $0<\gamma_2<\gamma_1<\min\{N, 4\}$, $\gamma_1>2$ and $N \geq 3$ and let $u_c \in S(c)$ be a solution to the equation
\begin{align} \label{equ11}
-\Delta u_c + \lambda_c u_c=\left(|x|^{-\gamma_1} \ast |u_c|^2\right) u_c - \left(|x|^{-\gamma_2} \ast |u_c|^2\right) u_c\quad \mbox{in} \,\, \R^N,
\end{align}
then there exists a constant $\hat{c}_1>0$ depending only on $N, \gamma_1, \gamma_2$ such that $\lambda_c>0$ for any $0<c<\hat{c}_1$.
\end{lem}
\begin{proof}
Multiplying \eqref{equ11} by $u_c$ and integrating on $\R^N$, we first get that
\begin{align} \label{ph11}
\hspace{-1cm}\int_{\R^N}|\nabla u_c|^2 \,dx + \lambda_c c + \int_{\R^N} \int_{\R^N}\frac{|u_c(x)|^2|u_c(y)|^2}{|x-y|^{\gamma_2}} \, dxdy=\int_{\R^N} \int_{\R^N}\frac{|u_c(x)|^2|u_c(y)|^2}{|x-y|^{\gamma_1}} \, dxdy.
\end{align}
In addition, from Lemma \ref{ph}, we have that
\begin{align} \label{ph12}
\hspace{-1cm}\int_{\R^N}|\nabla u_c|^2 \,dx + \frac {\gamma_2}{4}\int_{\R^N} \int_{\R^N}\frac{|u_c(x)|^2|u_c(y)|^2}{|x-y|^{\gamma_2}} \, dxdy=\frac{\gamma_1}{4} \int_{\R^N} \int_{\R^N}\frac{|u_c(x)|^2|u_c(y)|^2}{|x-y|^{\gamma_1}} \, dxdy.
\end{align}
Combining \eqref{ph11} and \eqref{ph12}, we then deduce that
\begin{align} \label{ph13}
\begin{split}
\lambda_c c&= \frac{4-\gamma_2}{\gamma_2} \int_{\R^N} |\nabla u_c|^2 \, dx + \frac{\gamma_2-\gamma_1}{\gamma_2}\int_{\R^N} \int_{\R^N}\frac{|u_c(x)|^2|u_c(y)|^2}{|x-y|^{\gamma_1}} \, dxdy \\
& \geq \frac{4-\gamma_2}{\gamma_2} \int_{\R^N} |\nabla u_c|^2 \, dx- \frac{C_{N, \gamma_1}\left(\gamma_1-\gamma_2\right)}{\gamma_2} \left( \int_{\R^N} |\nabla u_c|^2 \, dx \right)^{\frac{\gamma_1}{2}} c^{\frac{4-\gamma_1}{2}},
\end{split}
\end{align}
where we used \eqref{inequ} for the inequality, because of $\gamma_2<\gamma_1$. On the other hand, from \eqref{ph12} and \eqref{inequ}, we assert that
\begin{align*} 
\hspace{-1cm}\int_{\R^N}|\nabla u_c|^2 \,dx + \frac {\gamma_2}{4}\int_{\R^N} \int_{\R^N}\frac{|u_c(x)|^2|u_c(y)|^2}{|x-y|^{\gamma_2}} \, dxdy \leq \frac{C_{N, \gamma_1}\gamma_1}{4} \left( \int_{\R^N} |\nabla u_c|^2 \, dx \right)^{\frac{\gamma_1}{2}} c^{\frac{4-\gamma_1}{2}}.
\end{align*}
Since $\gamma_1>2$,  we then see that
$$
\int_{\R^N} |\nabla u_c|^2 \, dx  \geq \left(\frac{4}{C_{N, \gamma_1} c^{\frac{4-\gamma_1}{2}}}\right)^{\frac{2}{\gamma_1-2}}\to \infty \quad \mbox{as} \,\, c \to 0^+.
$$
Using \eqref{ph13}, we then have that $\lambda_c>0$ for any $c>0$ small enough, because of $2<\gamma_1<4$. Thus we have completed the proof.
\end{proof}

\begin{lem} \label{nonincreasing}
Let $0<\gamma_2<\gamma_1<\min\{N, 4\}$, $\gamma_1>2$ and $N \geq 3$. Then the function $c \mapsto \Gamma(c)$ is nonincreasing on $(0, \infty)$.
\end{lem}
\begin{proof}
To achieve this, we need to show that, if $0<c_2<c_1$, then $\Gamma(c_1) \leq \Gamma(c_2)$. By the definition of $\Gamma(c_2)$, we first have that, for any $\eps>0$, there exists $u \in P(c_2)$ such that
\begin{align} \label{non1}
E(u) \leq \Gamma(c_2) + \frac{\eps}{2}.
\end{align}
Let $\chi \in C_0^{\infty}(\R^N, [0,1])$ be a cut-off function such that $\chi(x) =1$ for $|x| \leq 1/2$, $\chi(x)=0$ for $|x|\geq 1$ and $|\nabla \chi(x)| \leq 4$ for any $x \in \R^N$. For $\delta>0$ small enough, we define $u^{\delta}(x):=u(x) \chi(\delta x)$ for $x \in \R^N$. One can check that $u^{\delta} \to u$ in $H^1(\R^N)$ as $\delta \to 0^+$. Since $Q(u)=0$, then $t_{u^{\delta}} \to 1$ as $\delta \to 0^+$, where $t_{u^{\delta}}>0$ is defined in Lemma \ref{unique} such that $Q((u^{\delta})_{t_{u^{\delta}}})=0$. As a result, we have that $(u^{\delta})_{t_{u^{\delta}}} \to u$ in $H^1(\R^N)$ as $\delta \to 0^+$. Therefore, there exists a constant $\delta>0$ small enough such that
\begin{align} \label{non2}
E((u^{\delta})_{t_{u_{\delta}}}) \leq E(u) + \frac{\eps}{4}.
\end{align}
Let $v \in C^{\infty}_0(\R^N)$ be such that $\mbox{supp}\,v \subset B(0, 1+2/\delta) \backslash B(0, 2/\delta)$ and set
$$
\tilde{v}^{\delta}:=\frac{\left(c_1-\|u^{\delta}\|_2^2\right)^{\frac 12}}{\|v\|_2} v.
$$ 
For $0<t<1$, we define $w^{\delta}_t:=u^{\delta} + (\tilde{v}^{\delta})_t$, where $ (\tilde{v}^{\delta})_t(x):= t^{N/2}\tilde{v}^{\delta}(tx)$ for $x\in \R^N$. Notice that 
$$
\textnormal{dist}\left(\mbox{supp}\, u^{\delta}, \, \mbox{supp}\, (\tilde{v}^{\delta})_t \right)=\frac{1}{\delta} \left(\frac 2 t-1\right)>\frac{1}{\delta},
$$ 
then $w^{\delta}_t \in S(c_1)$. Observe that 
$$
 \int_{\R^N}|\nabla(\tilde{v}^{\delta})_t|^2 \,dx ==o_t(1), \quad \int_{\R^N} \int_{\R^N}\frac{|(\tilde{v}^{\delta})_t(x)|^2|(\tilde{v}^{\delta})_t(y)|^2}{|x-y|^{\gamma_i}} \, dxdy=o_t(1),
$$
where $i=1,2$. Hence we are able to deduce that there exist $t>0$ small enough such that
\begin{align}\label{non3}
\max_{\lambda>0}E(((\tilde{v}^{\delta})_t)_{\lambda}) \leq \frac{\eps}{4}.
\end{align}
Consequently, with the aid of \eqref{non1}-\eqref{non3}, we have that
\begin{align*}
\Gamma(c_1) \leq \max_{\lambda>0} E((w^{\delta}_t)_{\lambda})&=\max_{\lambda>0}\left(E((u^{\delta})_{\lambda}) + E(((\tilde{v}^{\delta})_t)_{\lambda}) \right) +o_{\delta}(1) \\
&\leq E((u^{\delta})_{t_{u^{\delta}}})  + \frac{\eps}{4} +o_{\delta}(1)\leq E(u)+ \frac{\eps}{2} +o_{\delta}(1)\leq \Gamma(c_2) +\eps + o_{\delta}(1).
\end{align*}
This leads to $\Gamma(c_1) \leq \Gamma(c_2)$. Thus the proof is completed.
\end{proof}

\begin{lem}\label{decreasing}
Let $0<\gamma_2<\gamma_1<\min\{N, 4\}$, $\gamma_1>2$ and $N \geq 3$. If there exists $u \in S(c)$ with $E(u)=\Gamma(c)$ satisfying the equation
\begin{align} \label{fequ2}
-\Delta u+ \lambda u=\left(|x|^{-\gamma_1} \ast |u|^2\right) u - \left(|x|^{-\gamma_2} \ast |u|^2\right) u\quad \mbox{in} \,\, \R^N,
\end{align}
then $\lambda \geq 0$. If $\lambda>0$, then the function $c \mapsto \Gamma(c)$ is strictly decreasing in a right neighborhood of $c$.  
\end{lem}
\begin{proof}
Let $u \in S(c)$ be such that $E(u)=\Gamma(c)$ and $u$ solves \eqref{fequ2} for some $\lambda \in \R$. We shall prove that if $\lambda>0$, then the function $c \mapsto \Gamma(c)$ is strictly decreasing in a right neighborhood of $c$.  For any $t_1, t_2>0$, we define $u_{t_1,t_2}(x):=t_1^{1/2}t_2^{N/2}u(t_2x)$ for $x \in \R^N$. By simple calculations, we see that $\|u_{t_1,t_2}\|_2^2 =t_1c$ and
\begin{align*}
\alpha_u(t_1,t_2):=E(u_{t_1,t_2})&=\frac {t_1^2t_2^2}{2} \int_{\R^N}|\nabla u|^2 \,dx + \frac {t_1^4 t_2^{\gamma_2}}{4}\int_{\R^N} \int_{\R^N}\frac{|u(x)|^2|u(y)|^2}{|x-y|^{\gamma_2}} \, dxdy \\
& \quad-\frac {t_1^4t_2^{\gamma_1}}{4} \int_{\R^N} \int_{\R^N}\frac{|u(x)|^2|u(y)|^2}{|x-y|^{\gamma_1}} \, dxdy,
\end{align*}
\begin{align*}
\beta_u(t_1,t_2):=Q(u_{t_1,t_2})&={t_1^2t_2^2}\int_{\R^N}|\nabla u|^2 \,dx + \frac {\gamma_2t_1^4 t_2^{\gamma_2}}{4}\int_{\R^N} \int_{\R^N}\frac{|u(x)|^2|u(y)|^2}{|x-y|^{\gamma_2}} \, dxdy \\
& \quad-\frac {\gamma_1t_1^4t_2^{\gamma_1}}{4} \int_{\R^N} \int_{\R^N}\frac{|u(x)|^2|u(y)|^2}{|x-y|^{\gamma_1}} \, dxdy.
\end{align*}
We then compute that
\begin{align*}
\frac{\partial \alpha_u} {\partial t_1}(1,1)=-\lambda c, \quad   \frac{\partial \alpha_u} {\partial t_2}(1,1)=0, \quad \frac{\partial^2\alpha_u} {\partial t_2^2}(1,1)<0.
\end{align*}
Therefore, there exist $\delta_1>0$ and $|\delta_2|$ small enough such that
$$
\alpha(1+\delta_1, 1+\delta_2)<\alpha(1, 1) \quad \mbox{if} \,\, \lambda>0,
$$
$$
\alpha(1-\delta_1, 1+\delta_2)<\alpha(1, 1) \quad \mbox{if} \,\, \lambda<0.
$$
Observe that $\beta_u(1,1)=0$ and $\frac{\partial \beta_u}{\partial t_2} (1,1)\neq 0$. Using the implicit function theorem, we then know that there exist $\eps>0$ small enough and $g \in C([1-\eps, 1+\eps], \R)$ such that $g(1)=1$ and $\beta_u(t, g(t))=0$ for any $t \in [1-\eps, 1+\eps]$. As a result, we get that
\begin{align} \label{dc111}
\Gamma((1+\eps)c) =\inf_{u \in P((1+\eps)c )}\leq E(u_{1+\eps, g(1+\eps)})<E(u)=\Gamma(c) \quad \mbox{if} \,\, \lambda>0,
\end{align}
\begin{align} \label{dc112}
\Gamma((1-\eps)c) =\inf_{u \in P((1-\eps)c )}\leq E(u_{1-\eps, g(1-\eps)})<E(u)=\Gamma(c) \quad \mbox{if} \,\, \lambda<0.
\end{align}
From \eqref{dc111} and Lemma \ref{nonincreasing}, then the function $c \mapsto \Gamma(c)$ is strictly decreasing in a  right neighborhood of $c$. If not, we then reach a contradiction from \eqref{dc111}. Making use of \eqref{dc112} and Lemma \ref{nonincreasing}, we then obtain that $\lambda \geq 0$. Thus the proof is completed.
\end{proof}

As a direct consequence of Lemmas \ref{sign} and \ref{decreasing}, we can derive the following result.

\begin{lem}\label{ladecreasing}
Let $0<\gamma_2<\gamma_1<\min\{N, 4\}$, $\gamma_1>2$ and $N \geq 3$.Then the function $c \mapsto \gamma(c)$ is strictly decreasing on $(0, \hat{c}_1)$, where the constant $\hat{c}_1>0$ is determined in Lemma \ref{sign}.
\end{lem}

\begin{proof}[Proof of Theorem \ref{thm1}]
From Lemma \ref{pss}, we know that there exists a Palais-Smale sequence $\{u_n\} \subset P(c)$ for $E$ restricted on $S(c)$ at the level $\Gamma(c)$. In view of Lemma \ref{below}, we have that $\{u_n\}$ is bounded in $H^1(\R^N)$. Reasoning as the proof of \cite[Lemma 3]{BL}, then we are able to deduce that $u_n$ satisfies the following equation,
\begin{align} \label{equ31}
-\Delta u_n+ \lambda_n u_n=\left(|x|^{-\gamma_1} \ast |u_n|^2\right) u_n - \left(|x|^{-\gamma_2} \ast |u_n|^2\right) u_n+o_n(1),
\end{align}
where
$$
\lambda_n:=\frac{1}{c} \left(\int_{\R^N} \int_{\R^N}\frac{|u_n(x)|^2|u_n(y)|^2}{|x-y|^{\gamma_1}} \, dxdy -\int_{\R^N} \int_{\R^N}\frac{|u_n(x)|^2|u_n(y)|^2}{|x-y|^{\gamma_2}} \, dxdy-\int_{\R^N} |\nabla u_n|^2\, dx\right).
$$
We now claim that $\{u_n\}$ is non-vanishing in $H^1(\R^N)$. Otherwise, taking into account \cite[Lemma I.1]{Li2}, we can derive that
$$
\int_{\R^N} \int_{\R^N}\frac{|u_n(x)|^2|u_n(y)|^2}{|x-y|^{\gamma_2}} \, dxdy =\int_{\R^N} \int_{\R^N}\frac{|u_n(x)|^2|u_n(y)|^2}{|x-y|^{\gamma_1}} \, dxdy=o_n(1).
$$
Since $Q(u_n)=0$, then $E(u_n)=o_n(1)$. This is impossible, because of $\Gamma(c)>0$, see Lemma \ref{below}. Therefore, we conclude that there exists a sequence $\{y_n\} \subset \R^N$ such that $u_n(\cdot+y_n) \wto u$  in $H^1(\R^N)$ as $n \to \infty$ and $u\neq 0$. Note that $\{u_n\}$ is bounded in $H^1(\R^N)$, then $\{\lambda_n\} \subset \R$ is bounded. This then shows that there exists a constant $\lambda\in\R$ such that $\lambda_n \to \lambda$ in $\R$ as $n \to \infty$. Consequently, from \eqref{equ31}, we get that $u \in H^1(\R^N)$ enjoys the equation
\begin{align} \label{fequ4}
-\Delta u+ \lambda u=\left(|x|^{-\gamma_1} \ast |u|^2\right) u - \left(|x|^{-\gamma_2} \ast |u|^2\right) u.
\end{align}
Using Lemma \ref{ph}, we then have that $Q(u)=0$. We now prove that $u \in S(c)$. To do this, we define $w_n=u_n-u(\cdot-y_n)$. Hence $w_n(\cdot+y_n) \wto 0$ in $H^1(\R^N)$ as $n \to \infty$. In addition, there holds that
\begin{align*}
\| \nabla u_n\|^2_2=\| \nabla w_n\|^2_2+\| \nabla  u\|^2_2+o_n(1)
\end{align*}\
and
$$
\int_{\R^N} \int_{\R^N}\frac{|u_n(x)|^2|u_n(y)|^2}{|x-y|^{\gamma_i}} \, dxdy=\int_{\R^N} \int_{\R^N}\frac{|w_n(x)|^2|w_n(y)|^2}{|x-y|^{\gamma_i}} \, dxdy +\int_{\R^N} \int_{\R^N}\frac{|u(x)|^2|u(y)|^2}{|x-y|^{\gamma_i}} \, dxdy+o_n(1),
$$
where $i=1,2$.
This then indicates that
\begin{align} \label{bl}
\hspace{-0.5cm}\Gamma(c)=E(u_n)+o_n(1)=E(w_n) +E(u)+o_n(1) \geq E(w_n) +\Gamma(\|u\|_2^2)+o_n(1)
\end{align}
and
\begin{align} \label{bl1}
Q(w_n)=Q(w_n)+Q(u)+o_n(1)=Q(u_n)+o_n(1)=o_n(1).
\end{align}
In view of \eqref{bl1} and $\gamma_1>2$, then
\begin{align}\label{bl11}
0 \leq E(w_n)-\frac{1}{\gamma_1}Q(w_n)=E(w_n) +o_n(1).
\end{align}
Since $0<\|u\|_2^2 \leq c$, by using \eqref{bl}, \eqref{bl11} and Lemma \ref{nonincreasing}, we then have that $\Gamma(\|u\|_2^2)=\Gamma(c)$. This together with Lemmas \ref{sign} and \ref{ladecreasing} infers the desired result. Thus the proof is completed.
\end{proof}

\section{Mass subcritical-critical case} \label{section5}

In this section, we investigate the mass subcritical-critical case $0<\gamma_2<\gamma_1=2$ and $N \geq 3$. And we shall prove Theorems \ref{nonexistence1} and \ref{thmcritical}.

\begin{proof} [Proof of Theorem \ref{nonexistence1}]
According to \eqref{inequ}, we first deduce that
\begin{align*}
E(u) \geq \frac 12 \left(1-\frac{C_{N,2} c}{2}\right) \int_{\R^N}|\nabla u|^2 \, dx + \int_{\R^N} \int_{\R^N}\frac{|u(x)|^2|u(y)|^2}{|x-y|^{\gamma_2}} \, dxdy,
\end{align*}
from which we have that $m(c) \geq 0$ for any $0<c\leq \tilde{c}_1:=2/C_{N,2}$. On the other hand, from \eqref{scaling}, we get that $E(u_t) \to 0$ as $t \to 0^+$. This infers that $m(c) \leq 0$ for any $c>0$. As a result, we have that $m(c)=0$ for any $0<c\leq \hat{c_1}$. We now show that $m(c)=-\infty$ for any $c>\tilde{c}_1$. To do this, we define
$$
w:=\frac{Q_2}{\|Q_2\|_2} c^{\frac 12} \in S(c),
$$
where $Q_2 \in H^1(\R^N)$ is a ground state to \eqref{equ2} with $\gamma=2$. By simple computations, then there holds that
\begin{align*}
E(w_t)=\frac{ct^2}{2\|Q_2\|_2^2} \left(1- \frac{C_{N, 2} c}{2}\right) \int_{\R^N} |\nabla u|^2\, dx + \frac{c^2t^{\gamma_2}}{4\|Q_2\|_2^4}\int_{\R^N} \int_{\R^N}\frac{|Q_2(x)|^2|Q_2(y)|^2}{|x-y|^{\gamma_2}} \, dxdy.
\end{align*}
Then we find that $E(w_t) \to -\infty$ as $t \to \infty$ for any $c>\tilde{c}_1$, because of $0<\gamma_2<2$. Hence we have the desired result. We now prove that there exists no solutions to \eqref{equ}-\eqref{mass} for any $0<c \leq \tilde{c}_1$. If there were a solution to \eqref{equ}-\eqref{mass} for some $0<c \leq \tilde{c}_1$, by Lemma \ref{ph}, then $u$ enjoys the following Pohozaev identity,
\begin{align*}
\int_{\R^N}|\nabla u|^2 \,dx + \frac {\gamma_2}{4}\int_{\R^N} \int_{\R^N}\frac{|u(x)|^2|u(y)|^2}{|x-y|^{\gamma_2}} \, dxdy&=\frac {1}{2}\int_{\R^N} \int_{\R^N}\frac{|u(x)|^2|u(y)|^2}{|x-y|^{2}} \, dxdy \\
& \leq \frac{C_{N,2} c}{2} \int_{\R^N}|\nabla u|^2 \,dx \leq \int_{\R^N}|\nabla u|^2 \,dx,
\end{align*}
where we used \eqref{inequ} in the first inequality. This is impossible. Thus the proof is completed.
\end{proof}

In order to seek for solutions to \eqref{equ}-\eqref{mass} for $c>\tilde{c}_1$, we introduce the following minimization problem,
$$
\Gamma(c):=\inf_{u \in \mathcal{P}(c)} E(u),
$$
where $\mathcal{P}(c):=P(c) \cap \mathcal{S}(c)$ and 
$$
\mathcal{S}(c):=\left\{ u \in S(c) : \int_{\R^N} |\nabla u|^2 \, dx <\frac 12\int_{\R^N} \int_{\R^N}\frac{|u(x)|^2|u(y)|^2}{|x-y|^{2}} \, dxdy \right\}.
$$
Note that $c>\tilde{c}_1$, then $\mathcal{S}(c) \neq \emptyset$. Indeed, one can check that
$$
w:=\frac{Q_2}{\|Q_2\|_2} c^{\frac 12} \in \mathcal{S}(c).
$$

\begin{lem} \label{below1}
Let $0<\gamma_2<\gamma_1=2$ and $N \geq 3$. Then $\Gamma(c)>0$ and $\Gamma(c)$ restricted on $\mathcal{P}(c)$ is coercive for any $c>\tilde{c}_1$.
\end{lem}
\begin{proof}
We first prove that $\Gamma(c)>0$ for any $c>\tilde{c}_1$. For any $u \in \mathcal{P}(c)$, we see that
\begin{align}\label{b111}
E(u)=E(u)-\frac 12 Q(u)=\frac{2-\gamma_2}{8}  \int_{\R^N} \int_{\R^N}\frac{|u(x)|^2|u(y)|^2}{|x-y|^{\gamma_2}} \, dxdy.
\end{align}
Arguing as the proof of Lemma \ref{estimate}, we can similarly deduce that there exists a constant $C>0$ depending only on $\gamma_2$ and $N$ such that
\begin{align*}
\int_{\R^N}\int_{\R^N}\frac{|u(x)|^2|u(y)|^2}{|x-y|^2} \, dxdy &\leq C\left(\int_{\R^N}\int_{\R^N}\frac{|u(x)|^2|u(y)|^2}{|x-y|^{\gamma_2}} \, dxdy\right)^{\theta} \left(\int_{\R^N}|\nabla u|^2\, dx\right)^{\frac{\gamma(1-\theta)}{2}} c^\frac{(4-\gamma)(1-\theta)}{2},
\end{align*}
where $2<\gamma<N$ and $0<\theta<1$ such that $2=\theta \gamma_2+(1-\theta)\gamma$. Using Young's inequality, we then have that
\begin{align*}
\int_{\R^N}\int_{\R^N}\frac{|u(x)|^2|u(y)|^2}{|x-y|^2} \, dxdy &\leq C_{\gamma_2, N}\left(\int_{\R^N}\int_{\R^N}\frac{|u(x)|^2|u(y)|^2}{|x-y|^{\gamma_2}} \, dxdy\right)^{\frac{2}{\gamma_2}}  c^\frac{(4-\gamma)(1-\theta)}{\theta \gamma_2} +2\int_{\R^N}|\nabla u|^2 \, dx.
\end{align*}
This gives that, for any $u \in \mathcal{P}$,
\begin{align*}
\int_{\R^N}|\nabla u|^2 \,dx + \frac {\gamma_2}{4}\int_{\R^N} \int_{\R^N}\frac{|u(x)|^2|u(y)|^2}{|x-y|^{\gamma_2}} \, dxdy&\leq  C_{\gamma_2, N}\left(\int_{\R^N}\int_{\R^N}\frac{|u(x)|^2|u(y)|^2}{|x-y|^{\gamma_2}} \, dxdy\right)^{\frac{2}{\gamma_2}}  c^\frac{(4-\gamma)(1-\theta)}{\theta \gamma_2} \\
& \quad +\int_{\R^N}|\nabla u|^2 \, dx.
\end{align*}
Since $0<\gamma_2<2$, then
$$
\int_{\R^N}\int_{\R^N}\frac{|u(x)|^2|u(y)|^2}{|x-y|^{\gamma_2}} \, dxdy \geq \left(\frac{\gamma_2}{4C_{\gamma_2, N} c^{\frac{(4-\gamma)(1-\theta)}{\theta \gamma_2}}}\right)^{\frac{\gamma_2}{2-\gamma_2}}.
$$
It then follows from \eqref{b111} that the desired result is valid. We next show that $E$ restricted on $\mathcal{P}(c)$ is coercive for any $c>\tilde{c}_1$. Let $\{u_n\} \subset \mathcal{P}(c)$ be such that $\|\nabla u_n\|_2 \to \infty$ as $n \to \infty$ for $c>\tilde{c}_1$, then
\begin{align} \label{l1}
\int_{\R^N}\int_{\R^N}\frac{|u_n(x)|^2|u_n(y)|^2}{|x-y|^{\gamma_2}} \, dxdy \to \infty \quad \mbox{as} \,\, n \to \infty.
\end{align}
If \eqref{l1} is true, it then yields from \eqref{b111} that $E(u_n) \to \infty$ as $n \to \infty$ and the proof is completed. Otherwise, there exist a sequence $\{u_n\} \subset \mathcal{P}(c)$ with $\|\nabla u_n\|_2 \to \infty$ as $n \to \infty$ and a constant $C>0$ such that
\begin{align} \label{bdd}
\left|\int_{\R^N}\int_{\R^N}\frac{|u_n(x)|^2|u_n(y)|^2}{|x-y|^{\gamma_2}} \, dxdy \right| \leq C.
\end{align}
Reasoning as before, by using Young's inequality, we get that
$$
\int_{\R^N}\int_{\R^N}\frac{|u_n(x)|^2|u_n(y)|^2}{|x-y|^2} \, dxdy \leq C\left(\int_{\R^N}\int_{\R^N}\frac{|u_n(x)|^2|u_n(y)|^2}{|x-y|^{\gamma_2}} \, dxdy\right)^{\frac{2}{\gamma_2}}  c^\frac{(4-\gamma)(1-\theta)}{\theta \gamma_2} +\frac 12 \int_{\R^N}|\nabla u_n|^2 \, dx.
$$
Since $Q(u_n)=0$, by applying \eqref{bdd}, we then obtain $\|\nabla u_n\|_2 \leq C $. We then reach a contradiction. Therefore, we have the desired conclusion and the proof is completed.
\end{proof}

\begin{lem} \label{sign1}
Let $0<\gamma_2<\gamma_1=2$ and $N \geq 3$ and let $u_c \in S(c)$ be a solution to the equation
\begin{align} \label{equ111}
-\Delta u_c + \lambda_c u_c=\left(|x|^{-2} \ast |u_c|^2\right) u_c - \left(|x|^{-\gamma_2} \ast |u_c|^2\right) u_c\quad \mbox{in} \,\, \R^N.
\end{align}
Then there exists a constant $c_1^*>0$ depending only on $N $ and $\gamma_2$ such that $\lambda_c>0$ for any $\tilde{c}_1<c<c_1^*$.
\end{lem}
\begin{proof}
Multiplying \eqref{equ111} by $u_c$ and integrating on $\R^N$, we see that
\begin{align} \label{ph111}
\hspace{-1cm}\int_{\R^N}|\nabla u_c|^2 \,dx + \lambda_c c + \int_{\R^N} \int_{\R^N}\frac{|u_c(x)|^2|u_c(y)|^2}{|x-y|^{\gamma_2}} \, dxdy=\int_{\R^N} \int_{\R^N}\frac{|u_c(x)|^2|u_c(y)|^2}{|x-y|^2} \, dxdy.
\end{align}
Furthermore, via Lemma \ref{ph}, we know that
\begin{align} \label{ph121}
\hspace{-1cm}\int_{\R^N}|\nabla u_c|^2 \,dx + \frac {\gamma_2}{4}\int_{\R^N} \int_{\R^N}\frac{|u_c(x)|^2|u_c(y)|^2}{|x-y|^{\gamma_2}} \, dxdy=\frac{\gamma_1}{4} \int_{\R^N} \int_{\R^N}\frac{|u_c(x)|^2|u_c(y)|^2}{|x-y|^2} \, dxdy.
\end{align}
Hence we obtain that
\begin{align} \label{ph131}
\begin{split}
\lambda_c c&= \frac{4-\gamma_2}{\gamma_2} \int_{\R^N} |\nabla u_c|^2 \, dx + \frac{\gamma_2-2}{\gamma_2}\int_{\R^N} \int_{\R^N}\frac{|u_c(x)|^2|u_c(y)|^2}{|x-y|^2} \, dxdy \\
& \geq \frac{4-\gamma_2-C_{N, 2}(2-\gamma_2)c}{\gamma_2} \int_{\R^N} |\nabla u_c|^2 \, dx,
\end{split}
\end{align}
where the inequality benefits from \eqref{inequ}. On the other hand, from \eqref{ph121} and \eqref{inequ}, we can deduce that
\begin{align} \label{ph141}
\hspace{-1cm}\int_{\R^N}|\nabla u_c|^2 \,dx + \frac {\gamma_2}{4}\int_{\R^N} \int_{\R^N}\frac{|u_c(x)|^2|u_c(y)|^2}{|x-y|^{\gamma_2}} \, dxdy \leq \frac{C_{N,2} c}{2}\int_{\R^N} |\nabla u_c|^2 \, dx.
\end{align}
Note that $C_{N, 2} c \to 2$ as $c \to \tilde{c}_1^+$, by the definition of $C_{N,2}$. According to \eqref{ph141}, we then get that $\|\nabla u_c\|_2 \to \infty$ as $c \to \tilde{c}_1^+$. Otherwise, \eqref{ph141} leads to
$$
\int_{\R^N} \int_{\R^N}\frac{|u_c(x)|^2|u_c(y)|^2}{|x-y|^{\gamma_2}} \, dxdy  \to 0 \quad \mbox{as} \, \, c \to \tilde{c}_1^+.
$$
This then implies that $\Gamma(c) \to 0$ as $c \to \tilde{c}_1^+$. This is impossible. Taking into account \eqref{ph131}, we then conclude the result of lemma. Thus the proof is completed.
\end{proof}

\begin{proof} [Proof of Theorem \ref{thmcritical}]
Replacing the roles of $S(c)$ by $\mathcal{S}(c)$ and using the same strategies, we can also deduce that Lemmas \ref{unique}, \ref{ps}, \ref{pss}, \ref{nonincreasing}  and \ref{ladecreasing} remain valid for $c>\tilde{c}_1$. Then there exists a Palais-Smale sequence $\{u_n\} \subset \mathcal{P}(c)$ for $E$ restricted on $\mathcal{S}(c)$ at the level $\Gamma(c)$. Applying Lemmas \ref{below1} and \ref{sign1} and discussing as the proof of Theorem \ref{thm1}, we then derive the result of the theorem and the proof is completed.
\end{proof}

\section{Zero mass equations} \label{section6}

In this section, we are going to establish Propositions \ref{asym} and \ref{l2}. To do this, we need to study solutions to \eqref{equz}.

\begin{lem} \label{x}
Let $0<\gamma_2<\gamma_1<4$ and $N \geq 5$. Then there exists a constant $C>0$ depending only on $\gamma_1, \gamma_2$ and $N$ such that
\begin{align*} 
\int_{\R^N}\int_{\R^N}\frac{|u(x)|^2|u(y)|^2}{|x-y|^{\gamma_1}} \, dxdy &\leq C\left(\int_{\R^N}\int_{\R^N}\frac{|u(x)|^2|u(y)|^2}{|x-y|^{\gamma_2}} \, dxdy\right)^{\theta} \left(\int_{\R^N}|\nabla u|^2 \, dx \right)^{2(1-\theta)},
\end{align*}
where $0<\theta<1$ and $\gamma_1=\theta \gamma_2+4(1-\theta)$.
\end{lem}
\begin{proof}
Taking into account H\"older's inequality, \eqref{HLS} and Sobolev's inequality, we deduce that
\begin{align*}
\int_{\R^N}\int_{\R^N}\frac{|u(x)|^2|u(y)|^2}{|x-y|^{\gamma_1}} \, dxdy &\leq \left(\int_{\R^N}\int_{\R^N}\frac{|u(x)|^2|u(y)|^2}{|x-y|^{\gamma_2}} \, dxdy\right)^{\theta}\left(\int_{\R^N}\int_{\R^N}\frac{|u(x)|^2|u(y)|^2}{|x-y|^4} \, dxdy\right)^{1-\theta} \\
& \leq C \left(\int_{\R^N}\int_{\R^N}\frac{|u(x)|^2|u(y)|^2}{|x-y|^{\gamma_2}} \, dxdy\right)^{\theta} \left( \int_{\R^N}|u|^{2^*}\, dx \right)^{\frac{4(1-\theta)}{2^*}} \\
& \leq C \left(\int_{\R^N}\int_{\R^N}\frac{|u(x)|^2|u(y)|^2}{|x-y|^{\gamma_2}} \, dxdy\right)^{\theta} \left( \int_{\R^N}|\nabla u|^2\, dx \right)^{2(1-\theta)} .
\end{align*}
This completes the proof.
\end{proof}

In order to establish the existence of solutions to \eqref{equz}, we introduce the following Sobolev space
$$
X:=\left\{ u\in D^{1,2}(\R^N) : \int_{\R^N}\frac{|u(x)|^2|u(y)|^2}{|x-y|^{\gamma_2}} \, dxdy<\infty \right\}
$$
equipped with the norm 
$$
\|u\|_X:=\left(\int_{\R^N} |\nabla u|^2 \, dx + \left(\int_{\R^N}\frac{|u(x)|^2|u(y)|^2}{|x-y|^{\gamma_2}} \, dxdy \right)^{\frac 12} \right)^{\frac 12}.
$$
The energy functional related to \eqref{equz} is given by
$$
E(u)=\frac 12 \int_{\R^N} |\nabla u|^2 \,dx + \frac 14 \int_{\R^N}\int_{\R^N}\frac{|u(x)|^2|u(y)|^2}{|x-y|^{\gamma_2}} \, dxdy-\frac 14 \int_{\R^N}\int_{\R^N}\frac{|u(x)|^2|u(y)|^2}{|x-y|^{\gamma_1}} \, dxdy.
$$
It follows from Lemma \ref{x} that $E$ is well-defined in $X$. Further, it is standard to check that $E$ is of class $C^1$ in $X$ and any critical point of $E$ is a solution to \eqref{equz}. We now define
\begin{align} \label{defm}
m:=\{ E(u) : u \in X \backslash\{0\}, Q(u)=0\}.
\end{align}

\begin{lem} \label{existencez}
Let $0<\gamma_2<\gamma_1<4$, $\gamma_1>2$ and $N \geq 5$, then the infimum $m>0$ is achieved and \eqref{equz} admits ground states in $X$.
\end{lem}
\begin{proof}
We first prove that $m>0$. If $Q(u)=0$, then
\begin{align} \label{belowm}
E(u)=E(u)-\frac {1}{\gamma_1} Q(u)=\frac{\gamma_1-2}{2\gamma_1} \int_{\R^N}|\nabla u|^2 \, dx + \frac{\gamma_1-\gamma_2}{4\gamma_1}  \int_{\R^N} \int_{\R^N}\frac{|u(x)|^2|u(y)|^2}{|x-y|^{\gamma_2}} \, dxdy.
\end{align}
On the other hand, if $Q(u)=0$, by Lemma \ref{x} and Young's inequality, then
\begin{align*}
\begin{split}
\int_{\R^N} |\nabla u|^2 \,dx + \frac{\gamma_2}{4}\int_{\R^N}\int_{\R^N}\frac{|u(x)|^2|u(y)|^2}{|x-y|^{\gamma_2}} \, dxdy&=\frac{\gamma_1}{4}\int_{\R^N}\int_{\R^N}\frac{|u(x)|^2|u(y)|^2}{|x-y|^{\gamma_1}} \, dxdy\\ 
& \leq \frac {\gamma_2}{8} \int_{\R^N}\int_{\R^N}\frac{|u(x)|^2|u(y)|^2}{|x-y|^{\gamma_2}} \, dxdy\\
& \quad +C \left(\int_{\R^N} |\nabla u|^2 \,dx\right)^2.
\end{split}
\end{align*}
This immediately yields that $\|\nabla u\|_2 \geq 1/C$.
Coming back to \eqref{belowm}, we then have that $m>0$, because of $0<\gamma_2<\gamma_1$ and $\gamma_1>2$. We next show that the infimum $m$ is achieved. Let $\{u_n\} \subset X$ be a minimizing sequence to \eqref{defm}, namely $Q(u_n)=0$ and $E(u_n)=m+o_n(1)$. Applying \eqref{belowm}, we obtain that $\{u_n\}$ is bounded in $X$. Moreover, for any $u \in X \backslash \{0\}$ and $Q(u)=0$, we find that 
$$
Q'(u)u=Q'(u)u-4Q(u)=-2\int_{\R^N} |\nabla u|^2 \,dx \neq 0.
$$
Therefore, by Ekeland variational principle, there exists a sequence $\{v_n\} \subset X$ with $\|v_n-u_n\|_X=o_n(1)$ and $Q(v_n)=0$ such that $E(v_n)=m+o_n(1)$ and $E'(v_n)=o_n(1)$. Then we have that $\{v_n\}$ is bounded in $X$.
Since $m>0$, by the same spirit of the proof of \cite[Lemma I.1]{Li2}, then there exists a sequence $\{y_n\} \subset \R^N$ and a nontrivial $v \in X$ such that $v(\cdot-y_n) \wto v$ in $X$ as $n \to \infty$. In addition, we see that $E'(v)=0$. Then we get that $Q(v)=0$. By the definition of $m$, then $E(v) \geq m$. Furthermore, we know that 
\begin{align*}
E(v)=E(v)-\frac {\gamma_1}{4}Q(v)&=\frac{\gamma_1-2}{2\gamma_1} \int_{\R^N}|\nabla v|^2 \, dx + \frac{\gamma_1-\gamma_2}{4\gamma_1}  \int_{\R^N} \int_{\R^N}\frac{|v(x)|^2|v(y)|^2}{|x-y|^{\gamma_2}} \, dxdy  \\
&\leq \liminf_{n \to \infty} 
\left(\frac{\gamma_1-2}{2\gamma_1} \int_{\R^N}|\nabla v_n|^2 \, dx + \frac{\gamma_1-\gamma_2}{4\gamma_1}  \int_{\R^N} \int_{\R^N}\frac{|v_n(x)|^2|v_n(y)|^2}{|x-y|^{\gamma_2}} \, dxdy\right) \\
&=\liminf_{n \to \infty} \left(E(v_n)-\frac {\gamma_1}{4}Q(v_n)\right)=\liminf_{n \to \infty}E(v_n)= m.
\end{align*}
As a consequence, we get that $E(v)=m$ and $m$ is achieved by $v \in X$. We now turn to assert that $v \in X$ is a solution to \eqref{equz}. Since $v \in X$ is a minimizer to \eqref{defm}, then there exists a constant $\mu \in \R$ such that
$$
-(1-2\mu)\Delta v +(1-\mu \gamma_2)\left(|x|^{-\gamma_1} \ast |v|^{2}\right) v =(1-\mu \gamma_1)\left(|x|^{-\gamma_2} \ast |v|^2\right) v\quad \mbox{in} \,\, \R^N.
$$
This leads to 
\begin{align} \label{phz}
\begin{split}
&(1-2\mu)\int_{\R^N} |\nabla v|^2 \,dx +  \frac{(1-\mu \gamma_2) \gamma_2}{4} \int_{\R^N}\int_{\R^N}\frac{|v(x)|^2|v(y)|^2}{|x-y|^{\gamma_2}} \, dxdy\\
& = \frac{(1-\mu \gamma_1) \gamma_1}{4} \int_{\R^N}\int_{\R^N}\frac{|v(x)|^2|v(y)|^2}{|x-y|^{\gamma_1}} \, dxdy.
\end{split}
\end{align}
If $\mu \neq 0$, it then follows from \eqref{phz} and $Q(v)=0$ that
\begin{align} \label{phz1}
\frac{\gamma_2}{2} \left(\frac{\gamma_2}{2}-1\right)\int_{\R^N}\int_{\R^N}\frac{|v(x)|^2|v(y)|^2}{|x-y|^{\gamma_2}} \, dxdy=
\frac{\gamma_1}{2} \left(\frac{\gamma_1}{2}-1\right)\int_{\R^N}\int_{\R^N}\frac{|v(x)|^2|v(y)|^2}{|x-y|^{\gamma_1}} \, dxdy.
\end{align}
If $0<\gamma_2<2$, we then reach a contradiction. This in turn implies that $\mu=0$. If $\gamma_2>2$, it then yields from \eqref{phz1} that
$$
\frac{\gamma_2}{4}\int_{\R^N}\int_{\R^N}\frac{|v(x)|^2|v(y)|^2}{|x-y|^{\gamma_2}} \, dxdy>\frac{\gamma_1}{4}\int_{\R^N} \int_{\R^N}\frac{|v(x)|^2|v(y)|^2}{|x-y|^{\gamma_1}} \, dxdy,
$$
where we used the assumption $\gamma_2<\gamma_1$. Using again the fact $Q(v)=0$, we then reach a contradiction. This readily means that $\mu=0$. Therefore, we know that $v \in X$ is a solution to \eqref{equz}. Thus the proof is completed.
\end{proof}

\begin{proof} [Proof of Proposition \ref{l2}]
Let us first prove that $u \in C^2(\R^N)$. Since $u \in X$, then $|x|^{-\gamma_1} \ast |u|^{2} \in \dot{H}^{s_1}(\R^N)$ and $|x|^{-\gamma_2} \ast |u|^{2} \in \dot{H}^{s_2}(\R^N)$, where $\dot{H}^{s_i}(\R^N)$ denote the homogeneous Sobolev space and $s_i:=(N-\gamma_i)/2$ for $i=1, 2$. Note that $\dot{H}^s(\R^N) \hookrightarrow L^{2^*_s}(\R^N)$, where $2^*_s:=2N/(N-2s)$. This means that  $|x|^{-\gamma_1} \ast |u|^{2} \in L^{2N/\gamma_1}(\R^N)$ and $|x|^{-\gamma_2} \ast |u|^{2} \in  L^{2N/\gamma_2}(\R^N)$. Since $0<\gamma_2<4$, from \cite[Theorem 1.1]{BGMMV}, we then have that $X \hookrightarrow L^p(\R^N)$ for any $p>0$ satisfying
\begin{align} \label{defp}
\frac{2(4+N-\gamma_2)}{2+N-\gamma_2}\leq p \leq 2^*.
\end{align} 
As a result, we obtain that $\left(|x|^{-\gamma_1} \ast |u|^{2}\right) u \in L^{q_1}(\R^N)$ and $\left(|x|^{-\gamma_2} \ast |u|^{2}\right) u \in L^{q_2}(\R^N)$ for any $q_1, q_2>0$ satisfying
$$
\frac{2N(4+N-\gamma_2)}{(N+2-\gamma_2)(N+2\gamma_1)} \leq q_1 \leq \frac{2N}{N-2+\gamma_1}, \quad 
\frac{2N(4+N-\gamma_2)}{(N+2-\gamma_2)(N+2\gamma_2)} \leq q_2 \leq \frac{2N}{N-2+\gamma_2}.
$$
We now write \eqref{equz} as
\begin{align}  \label{equz1}
-\Delta u +u =u+\left(|x|^{-\gamma_1} \ast |u|^{2}\right) u - \left(|x|^{-\gamma_2} \ast |u|^2\right) u\quad \mbox{in} \,\, \R^N.
\end{align}
Therefore, we have that 
$$
-\Delta u +u \in L^q(\R^N), \quad \frac{2(4+N-\gamma_2)}{2+N-\gamma_2} \leq q \leq \frac{2N}{N-2+\gamma_1}.
$$
This infers that $u \in W^{2, q}(\R^N)$. Since $N-6+\gamma_1 > 0$, 
it then follows that $u \in L^{p_0}(\R^N)$ for any $p_0>0$ satisfying  
$$
\frac{2(4+N-\gamma_2)}{2+N-\gamma_2}\leq p_0 \leq \frac{2N}{N-6+\gamma_1}.
$$
Then we further have that $\left(|x|^{-\gamma_1} \ast |u|^{2}\right) u \in L^{q_{1,0}}(\R^N)$ and $\left(|x|^{-\gamma_2} \ast |u|^{2}\right) u \in L^{q_{2,0}}(\R^N)$ for any $q_{1,0}, q_{2,0}>0$ satisfying
$$
\frac{2N(4+N-\gamma_2)}{(N+2-\gamma_2)(N+2\gamma_1)} \leq q_{1,0} \leq \frac{2N}{N-6+2\gamma_1}, \quad 
\frac{2N(4+N-\gamma_2)}{(N+2-\gamma_2)(N+2\gamma_2)} \leq q_{2,0} \leq \frac{2N}{N-6+\gamma_1+\gamma_2}.
$$
In light of \eqref{equz1}, we then get that 
$$
-\Delta u +u \in L^{q_0}(\R^N), \quad \frac{2(4+N-\gamma_2)}{2+N-\gamma_2} \leq q_0 \leq \frac{2N}{N-6+2\gamma_1}.
$$
This leads to $u \in W^{2, q_0}(\R^N)$. If $N-10+2\gamma_1\leq 0$, then $u \in L^{p_1}(\R^N)$ for any $p_1>0$ satisfying
$$
\frac{2(4+N-\gamma_2)}{2+N-\gamma_2} \leq p_1 <\infty.
$$
Otherwise, it yields that
$$
\frac{2(4+N-\gamma_2)}{2+N-\gamma_2} \leq p_1 \leq  \frac{2N}{N-10+2\gamma_1}.
$$
By induction, we then find that if $2q_{k-1}  \geq N$, then $u \in L^{p_k}(\R^N)$ for any $p_k>0$ satisfying
$$
\frac{2(4+N-\gamma_2)}{2+N-\gamma_2} \leq p_k <\infty.
$$
Otherwise, we conclude that
$$
\frac{2(4+N-\gamma_2)}{2+N-\gamma_2} \leq p_k \leq  \frac{Nq_{k-1}}{N-2q_{k-1}}.
$$
Observe that
$$
\frac{Nq_{k-1}}{N-2q_{k-1}}-p_{k-1}=\frac{2q^2_{k-1}}{N-2q_{k-1}}>0.
$$
We then derive that $u \in L^{p}(\R^N)$ for any $p>0$ satisying 
$$
\frac{2(4+N-\gamma_2)}{2+N-\gamma_2} \leq p <\infty.
$$
This in turn infers that $u \in C^2(\R^N)$. In addition, since $C_0^{\infty}(\R^N) \hookrightarrow W^{2, p}(\R^N)$ for $2p>N$, then $u(x) \to 0$ as $|x| \to \infty$.

Next we prove the desired decay properties of solutions. For this, let us first deduce that $|u(x)| \leq C |x|^{2-N}$ for $|x|>0$ large enough. To do so, we make use of some ideas from \cite{EW}.
By using H\"older's inequality, we first have that
\begin{align*}
\int_{\R^N} |x-y|^{-\gamma_i} |u(y)|^2 \,dy &= \int_{B_R(x)} |x-y|^{-\gamma_i} |u(y)|^2 \,dy + \int_{\R^N \backslash B_R(x)} |x-y|^{-\gamma_i} |u(y)|^2 \,dy\\
& \leq \sup_{y \in B_R(x)}  |u(y)|^2 \int_{B_R(x)} |x-y|^{-\gamma_i} \,dy \\
& \quad + \left(\int_{\R^N \backslash B_R(x)} |u(y)|^p \,dx \right)^{\frac 1 p} \left(\int_{\R^N \backslash B_R(x)} |x-y|^{-\gamma p'}\,dy \right)^{\frac {1} {p'}},
\end{align*}
where $p':=p/(p-1)$ and $p>1$ satisfies \eqref{defp}. Since $0<\gamma_i<N$ and $u(x) \to 0$ as $|x| \to \infty$, then
\begin{align*}
\lim_{|x| \to \infty}\int_{\R^N} |x-y|^{-\gamma_i} |u(y)|^2 \,dy =o_R(1),
\end{align*}
where $o_R(1) \to 0$ as $R \to \infty$. 
By applying \eqref{HLS} and H\"older's inequality, we can similarly obtain that
\begin{align} \label{small}
\sup_{x \in \R^N}\int_{\R^N\backslash B_R(0)} |x-y|^{2-N} (|x|^{-\gamma_i} \ast |u|^{2})(y) \,dy =o_R(1).
\end{align}
Note that $u \in X$ is a solution to \eqref{equz}, then 
$$
u(x)=\int_{\R^N} |x-y|^{2-N} \left((|x|^{-\gamma_1} \ast |u|^{2})(y)-(|x|^{-\gamma_2} \ast |u|^{2})(y)\right) u(y) \, dy.
$$
Let $h:=(|x|^{-\gamma_1} \ast |u|^{2})-(|x|^{-\gamma_2} \ast |u|^{2})$. Define
$$
u_0(x):=\int_{B_R(0)} |x-y|^{2-N} h(y) u(y)\, dy, \quad  u_k(x):=\int_{\R^N \backslash B_R(0)} |x-y|^{2-N} h(y) u_{k-1}(y) \, dy
$$
and
$$
B_0(x):=\int_{\R^N \backslash B_R(0)} |x-y|^{2-N} h(y)u(y)\, dy, \quad B_k(x):=\int_{\R^N \backslash B_R(0)} |x-y|^{2-N} h(y) B_{k-1}(y)\, dy, 
$$
where $k \in \N$ and $k \geq 0$. Observe that
\begin{align*}
B_0(x)-B_1(x)&=\int_{\R^N \backslash B_R(0)} |x-y|^{2-N}h(y)\left(u(y)-B_0(y)\right)\, dy\\
&=\int_{\R^N \backslash B_R(0)} |x-y|^{2-N}h(y)u_0(y)\, dy=u_1(x)
\end{align*}
and
\begin{align*}
B_1(x)-B_2(x)&=\int_{\R^N \backslash B_R(0)} |x-y|^{2-N}h(y)\left(B_0(y)-B_1(y)\right)\, dy \\
&=\int_{\R^N \backslash B_R(0)} |x-y|^{2-N}h(y)u_1(y)\, dy=u_2(x).
\end{align*}
By induction, we see that $B_k(x)-B_{k+1}(x)=u_{k+1}(x)$ for any $k \in \N$ and $k \geq 0$. Therefore, there holds that
$$
u=u_0+B_0= \sum_{k=0}^m u_k + B_m.
$$
Let $\beta_k:=\sup_{|x| \geq R} |B_k(x)|$. Form the definition, by H\"older's inequality, it is easy to obtain that $\beta_0<\infty$. Using \eqref{small} and taking $R>0$ large enough, we then have that $\beta_k \leq 2^{-N}\beta_{k-1} \leq 2^{-kN} \beta_0=o_k(1)$. This then leads to
\begin{align} \label{usum}
u=\sum_{k=0}^{\infty} u_k,
\end{align}
which holds uniformly in $\R^N \backslash B_R(0)$. Let $\mu_k:=\sup_{|x| \geq R} |x|^{N-2} |u_k(x)|$. Since, for any $|x| \geq 2R$, 
$$
|u_0(x)| \leq \int_{B_R(0)} |x-y|^{2-N} |h(y)u(y)|\, dy \leq C |x|^{2-N},
$$
then we infer that $\mu_0 <\infty$. Notice that, for any $|x| \geq R$,
\begin{align} \label{small1}
\begin{split}
|x|^{N-2}|u_k(x)| &\leq \mu_{k-1} |x|^{N-2} \int_{\R^N \backslash B_R(0)} |x-y|^{2-N} |y|^{2-N} h(y)\, dy \\
& \leq C \mu_{k-1} \int_{\R^N \backslash B_R(0)} |x-y|^{2-N} \left(1+|y|^{2-N}|x-y|^{N-2}\right) h(y)\, dy \\
& \leq \frac{\mu_{k-1}}{2},
\end{split}
\end{align}
where we used \eqref{small} and chose $R>0$ large enough. By utilizing \eqref{usum} and iterating \eqref{small1}, we then obtain that, for any $|x| \geq R$,
$$
|x|^{N-2}|u(x)| \leq \sum_{k=0}^{\infty} |x|^{2-N} u_k(x) \leq \sum_{k=0}^{\infty} \mu_k \leq  \mu_0\sum_{k=0}^{\infty} 2^{-k}<\infty.
$$
Hence we derive that $|u(x)| \leq C|x|^{2-N}$ for $|x|>0$ large enough. Since $N \geq 5$, then $u \in L^2(\R^N)$. It then follows from \cite[Lemma 6.2]{MVS} that $|x|^{-\gamma_i} \ast |u|^{2} \sim |x|^{-\gamma_i}$ as $|x| \to \infty$. Taking $R>0$ large enough, we then have that
$$
-\Delta u + C\frac{u}{|x|^{\gamma_2}} \geq 0, \quad \mbox{in} \,\,\, \R^N \backslash B_R(0).
$$
If $0<\gamma_2<2$, by \cite[Theorem 3.3]{Ag} and \cite[Proposition B.2]{MVS1}, then
$$
u(x) \sim |x|^{-\frac{N-1}{2}+\frac{\gamma_2}{4}} e^{-\frac{2C^{\frac 12}}{2-\gamma_2} |x|^{1-\frac{\gamma_2}{2}}} \quad \mbox{as}\,\,\, |x| \to \infty.
$$
If $\gamma_2=2$, by \cite[Theorem 3.4]{LLM} and the maximum principle, then, for any $|x| \geq R$,
$$
u(x) \geq C|x|^{2-N}.
$$
This implies that $u(x) \sim |x|^{2-N}$ as $|x| \to \infty$. If $2<\gamma_2<4$, then $u$ satisfies 
$$
-\Delta u + V(x)u \geq 0, \quad \mbox{in} \,\,\, \R^N \backslash B_R(0),
$$
where $V(x):=C|x|^{-\gamma_2} \leq C(1+|x|)^{-2-\delta}$ for any $|x| \geq R$ and some $\delta>0$. By comparing with an explicit subsolution $C|x|^{2-N}(1+|x|^{-\delta/2})$, we conclude that, for any $|x| \geq R$,
$$
u(x) \geq C(1+|x|)^{2-N}.
$$
Therefore, we derive that $u(x) \sim |x|^{2-N}$ as $|x| \to \infty$. Thus the proof is completed.
\end{proof}

\begin{proof} [Proof of Proposition \ref{asym}]
Let $u_c \in S(c)$ be a solution to \eqref{equ}-\eqref{mass} such that $E(u_c)=\Gamma(c)$, see Theorem \ref{thm1}. Via Lemma \ref{ph}, we know that $Q(u_c)=0$. In view of \eqref{largenable1} and \eqref{largenable}, we then deduce that $\|\nabla u_c\|_2 \to + \infty$ as $c \to 0^+$. It then follows from \eqref{b1} that $\lim_{c \to 0^+}\Gamma(c)=+ \infty$. From the definitions of $\Gamma(c)$ and $m$, we conclude that $\Gamma(c) \geq m$ for any $c >0$. Then we get that $\lim_{c \to \infty} \Gamma(c) \geq m$. Let $u \in X$ be a solution to \eqref{equz}. If $N \geq 5$, then $u \in L^2(\R^N)$, see Proposition \ref{l2}. Define $c_{\infty}:=\|u\|_2^2$, then $\Gamma(c_{\infty}) \leq E(u)=m$. Therefore, we derive that $\Gamma(c)=m$ for any $c \geq c_{\infty}$ by Lemma \ref{nonincreasing}. Thus the proof is completed.
\end{proof}

\section{Local and global well-posedness} \label{section7}


In this section, we consider local and global well-posedness of solutions to \eqref{nlscn}. We first show the local well-posedness of solutions to \eqref{nlscn} in $H^s(\R^N)$ for $s \geq \frac{\gamma_1}{2}$, whose proof is inspired by \cite{Miao}. 
In order to discuss well-posedness of solutions to \eqref{nlscn}, we need to present some useful elements in the following.

\begin{lem} \label{gL} (\cite{K1})
For any $s\geq 0,$ we have that
\[ \| D^s (uv)\|_{L^r} \lesssim \| D^su\|_{L^{r_1}} \|v\|_{L^{q_2}} + \|u\|_{L^{q_1}}\|D^sv\|_{L^{r_2}},\]
where $D^s=(-\Delta)^s$, $\frac{1}{r}= \frac{1}{r_1}+ \frac{1}{q_2}=\frac{1}{q_1}+ \frac{1}{r_2}$ $r_1, r_2 \in (1, \infty)$and $q_1, q_2 \in (1, \infty]$.
 \end{lem}


\begin{lem} \label{mE} (\cite{K2})
For $0<\gamma < N,$ we have that
\[\|I_{N-\gamma}(|u|^2)\|_{L^{\infty}} \lesssim \| u\|^2_{ \dot{H}^{\frac{\gamma}{2}}}, \]
where $I_{\alpha}$ is the fractional integral operator $I_{\alpha}(v)(x)=\left(|\cdot|^{-(N-\alpha) }\ast v \right) (x).$
\end{lem}
\begin{lem} \label{SobE}
\begin{itemize}
\item[$(\textnormal{i})$] Let $s< \frac{N}{2}$ and $\frac{1}{p}= \frac{1}{2}- \frac{s}{N}$. Then there exists $C=C(N,p)$ such that  $\|f\|_{L^p}\leq C \|f\|_{\dot{H}^s}$ for all $f\in \dot{H}^s(\R^N)$.
\item[$(\textnormal{ii})$] Let $s< \frac{N}{2}$ and $\frac{1}{2}- \frac{s}{N} \leq \frac{1}{p}  \leq \frac 12$. Then there exists $C=C(N,p)$ such that  $\|f\|_{L^p}\leq C \|f\|_{{H}^s}$ for all $f\in H^s(\R^N)$.
\end{itemize}
\end{lem}

For convenience, we present a variant of \eqref{HLS} as follows.

\begin{lem} \label{hls} Let $0<\gamma< N$ and $1<p<q< \infty$ with
$$\frac{1}{p}+\frac{\gamma}{N}-1= \frac{1}{q}.$$
Then the map $f \mapsto |x|^{-\gamma}\ast f$ is bounded from $L^p(\mathbb R^N)$ to $L^q(\mathbb R^N):$
$$\|x|^{-\gamma}\ast f\|_{L^q} \lesssim \|f\|_{L^p}.$$
\end{lem}



\begin{proof}[Proof of Theorem \ref{LW}] Let $\left( X_{T, \rho}^s, d \right)$ be a complete metric space with metric $d$ defined by
\[X^{s}_{T, \rho}:= \left\{u \in L^{\infty}_TH^s(\R^N) : \|u\|_{L^{\infty}_TH^s} \leq \rho \right\}, \quad d(u,v) := \|u-v\|_{L^{\infty}_TL^2}.\]
Let us define a mapping $J:u \mapsto J(u)$ on $X^s_{T, \rho}$ by
\begin{align}
J(u)(t):= U(t)u_0-\textnormal{i} \int_0^t U(t-t')F(u)(t') \, dt',
\end{align}
where $U(t):=e^{-\textnormal{i} \Delta t}$.
In what follows, we shall prove that the mapping $J$ is a contraction on $X^{s}_{T,  \rho}$. Using Lemma \ref{hls}, we first get that
\begin{align*}
 \|I_{N-\gamma_i}(|u|^2)\|_{{H}^{s}_{\frac{2N}{\gamma_i}}} &= \|(I-\Delta)^{s/2}\left(I_{N-\gamma_i}(|u|^2)\right)\|_{L^{\frac{2N}{\gamma_i}}}= \|I_{N-\gamma_i} \left((I-\Delta)^{s/2}(|u|^2)\right)\|_{L^{\frac{2N}{\gamma_i}}}\\
 & \lesssim \|(I-\Delta)^{s/2} (|u|^2)\|_{L^{\frac{2N}{2N-\gamma_i}}}= \| |u|^2\|_{\dot{H}_{\frac{2N}{2N-\gamma_i}}^s}.
\end{align*}
Note that
\[\frac{1}{2}=\frac{1}{2}+ \frac{1}{\infty}, \quad \frac{1}{2}= \frac{\gamma_i}{2N}+\frac{N-\gamma_i}{2N} \quad \text{and} \quad \frac{2N-\gamma_i}{2N}= \frac{1}{2}+ \frac{N-\gamma_i}{2N}.\]
As a consequence of Lemmas \ref{gL} and  \ref{mE} and Lemma \ref{SobE}, we then have that
\begin{align}\label{ue}
\|J(u)\|_{L^{\infty}_T H^s} 
& \leq \|u_0\|_{H^s} + T \sum_{i \in \{1,2\}} \left( \|I_{N-\gamma_i}(|u|^2)\|_{L^{\infty}_TL^{\infty}} \|u\|_{L^{\infty}_T H^s} +  \|I_{N-\gamma_i}(|u|^2)\|_{L^{\infty}_T H^{s}_{\frac{2N}{\gamma_i}}} \|u\|_{L^{\infty}_T L^{\frac{2N}{N-\gamma_i}}} \right) \nonumber \\
& \lesssim  \|u_0\|_{H^s} + T \sum_{i \in \{1,2\}} \left( \|u\|^2_{L^{\infty}_T \dot{H}^{\frac{\gamma_i}{2}}} \|u\|_{L^{\infty}_TH^s} + \left\| |u|^2 \right\|_{L^{\infty}_T H^{s}_{\frac{2N}{2N-\gamma_i}}}   \|u\|_{L^{\infty}_T L^{\frac{2N}{N-\gamma_i}}} \right) \nonumber\\
& \lesssim  \|u_0\|_{H^s} + T \sum_{i \in \{1,2\}} \left( \|u\|^2_{L^{\infty}_T \dot{H}^{\frac{\gamma_i}{2}}} \|u\|_{L^{\infty}_TH^s} + \|u\|_{L^{\infty}_T H^s}  \|u\|^2_{L^{\infty}_T L^{\frac{2N}{N-\gamma_i}}} \right) \nonumber\\
& \lesssim  \|u_0\|_{H^s} + T \sum_{i \in \{1,2\}}  \|u\|^2_{L^{\infty}_T \dot{H}^{\frac{\gamma_i}{2}}} \|u\|_{L^{\infty}_TH^s} \nonumber\\
 & \lesssim  \|u_0\|_{H^s} + T\rho^3.   
\end{align}
If we choose $T$ and $\rho$ such that
\[\|u_0\|_{H^s} \leq \frac{\rho}{2}, \quad CT\rho^3 \leq \frac{\rho}{2},\]
then $J$ maps $X^s_{T, \rho}$ to itself.  Let $u, v\in X^{s}_{T, \rho}$, by H\"older's inequality and Lemma \ref{mE}, we  are able to derive that
\begin{align*}
d(J(u), J(v)) &  \lesssim    T  \sum_{i\in \{1,2\}} \left( \| I_{N-\gamma_i}(|u|^2)u- I_{N-\gamma_i}(|v|^2)v\|_{L^{\infty}_{T}L^2} \right)\\
 & \lesssim   T  \sum_{i\in \{1,2\}} \left( \| I_{N-\gamma_i}(|u|^2)(u-v)\|_{L^{\infty}_T L^2} + \| I_{N-\gamma_i}(|u|^2-|v|^2)v\|_{L^{\infty}_{T}L^2} \right)\\
 & \lesssim   T  \sum_{i\in \{1,2\}} \left( \|u\|_{L^{\infty}_T H^{\frac{\gamma}{2}}}^2 d(u,v) + \| I_{N-\gamma_i}(|u|^2-|v|^2)\|_{L^{\infty}_T L^{\frac{2N}{\gamma_i}}} \|v\|_{L^{\infty}_{T}L^{\frac{2N}{N-\gamma_i}}} \right)\\
  & \lesssim   T  \sum_{i\in \{1,2\}} \left( \|u\|_{L^{\infty}_T H^{\frac{\gamma}{2}}}^2 d(u,v) + \| I_{N-\gamma_i}(|u|^2-|v|^2)\|_{L^{\infty}_T L^{\frac{2N}{\gamma_i}}} \|v\|_{L^{\infty}_{T}L^{\frac{2N}{N-\gamma_i}}} \right)\\
  & \lesssim   T  \sum_{i\in \{1,2\}} \left( \rho^2 d(u,v) +  \rho \| (|u|^2-|v|^2)\|_{L^{\infty}_T L^{\frac{2N}{2N-\gamma_i}}}  \right)\\
  & \lesssim   T d(u,v)  \sum_{i\in \{1,2\}} \left( \rho^2  +  \rho \| u+v\|_{L^{\infty}_T L^{\frac{2N}{N-\gamma_i}}}  \right)\\
  & \lesssim  T\rho^2 d(u,v).
\end{align*}
This readily implies that $J$ is a contraction on $X^s_{T,\rho}$ if $T$ is sufficiently small. Now the proof follows by the standard Banach fixed point argument. For the proof of conservation laws, methods can be applicable as in \cite{L} and \cite{TO}. This completes the proof.
\end{proof}




\begin{proof} [Proof of Theorem \ref{GLW}]
If 
$$
\frac{1}{4} \int_{\R^N} \left(|x|^{-\gamma_1} \ast |u|^2\right) |u|^2 dx - \frac{1}{4} \int_{\R^N} \left(|x|^{-\gamma_2} \ast |u|^2\right) |u|^2 dx \geq 0,
$$ 
by the conservation laws,  then
$$
\|u\|^2_{H^{\frac{\gamma}{2}}}\leq \|u\|^2_{H^{1}}  \lesssim  E(u) + \|u_0\|_{L^2}^2 = E(u_0) + \|u_0\|_{L^2}^2.
$$
If 
$$\frac{1}{4} \int_{\R^N} \left(|x|^{-\gamma_1} \ast |u|^2\right) |u|^2 dx - \frac{1}{4} \int_{\R^N} \left(|x|^{-\gamma_2} \ast |u|^2\right) |u|^2 dx < 0,$$
by Lemma \ref{mE},  the conservation of mass and H\"older's inequality,  then
\begin{align*}
 \left| \frac{1}{4} \int_{\R^N} \left(|x|^{-\gamma_1} \ast |u|^2\right) |u|^2 dx - \frac{1}{4} \int_{\R^N} \left(|x|^{-\gamma_2} \ast |u|^2\right) |u|^2 dx\right| 
& \lesssim   \sum_{i\in \{ 1,2 \}} \|I_{N-\gamma_i}(|u|^2)\|_{L^{\infty}} \|u\|_{L^2}\\
&  \lesssim   \sum_{i\in \{ 1,2 \}} \|u\|^2_{\dot{H}^{\frac{\gamma_i}{2}}} \|u_0\|_{L^2}\\
&  \lesssim   \sum_{i\in \{ 1,2 \}} \|u\|_{\dot{H}^{1}}^{\gamma_i} \|u\|_{L^2}^{2-\gamma_i}     \|u_0\|^2_{L^2}\\
&  \lesssim    \sum_{i\in \{ 1,2 \}} \|u\|_{\dot{H}^{1}}^{\gamma_i}    \|u_0\|^{4-\gamma_i}_{L^2}\\
&  \lesssim   \sum_{i\in \{ 1,2 \}} \epsilon \|u\|_{\dot{H}^1}^2 + C(\epsilon) \| u_0\|_{L^2}^{\frac{8-2\gamma_i}{2-\gamma_i}}.
\end{align*}
This together with the conservation of energy leads to
$$
\|u(t)\|^2_{H^{\frac{\gamma}{2}}}\leq \|u(t)\|^2_{H^{1}} \lesssim \sum_{i \in \{1,2\} } \left(  E (u_0) + \|u_0\|_{L^2}^2 + \|u_0\|_{L^2}^{\frac{8-2\gamma_i}{2-\gamma_i}} \right).
$$
Then we have the desired conclusion. This completes the proof.
\end{proof}

In this subsection,  our aim is to investigate well-posedness of solutions to \eqref{nlscn} for relatively rough data in $H^{s}(\R^N)$ for $s< \frac{\gamma_1}{2}$. To this end,  we mainly employ  Strichartz estimates to run the fixed point argument.

\begin{defi} \label{defap}
 A pair $(q,r)$ is  admissible if  $q\geq 2, r\geq 2$ and
$$\frac{2}{q} =  N \left( \frac{1}{2} - \frac{1}{r} \right), \quad (q,r,N)\neq(\infty,2,2).$$
\end{defi}

\begin{lem}\label{st} (\cite{keeltao})
Let $\phi \in L^2(\R^N).$
Denote
$$DF(t, x) =  U(t)\phi(x) +  \int_0^t U(t-s)F(s,x) ds.$$
 Then for any time interval $I\ni0$ and admissible pairs $(q_j,r_j)$, $j=1,2,$
there exists  a constant $C=C(r_1,r_2)$ such that
$$ \|D(F)\|_{L^{q_1}(I,L^{r_1})}  \leq  C \|\phi \|_{L^2}+   C  \|F\|_{L^{q'_2}(I,L^{r'_2})}, \quad \, \forall F \in L^{q_2'} (I, L^{r_2'}(\R^N))$$  where $q_j'$ and $ r_j'$ are H\"older conjugates of $q_j$ and $r_j$, respectively.
\end{lem}
For $0<\gamma_2< \gamma_1<N$ and $\max\{ \max \{0, \frac{\gamma_i}{2}\}: i=1,2\} \leq s < \frac{\gamma_2}{2},$ there is a admissible pairs $(q_i, r_i)$ defined by
\begin{align}\label{yu}
\frac{1}{q_i}= \frac{\gamma_i-2s}{6}, \quad \frac{1}{r_i}=\frac{1}{2}+ \frac{2s-\gamma_i}{3N}.
\end{align}

\begin{proof}[Proof of Theorem \ref{LW1}]
Let $\left( Y_{T, \rho}^s, d \right)$ be a complete metric space with metric $d$ defined by
$$Y^{s}_{T, \rho}:= \left\{u \in L^{\infty}_TH^s\cap L^{q_1}_{T}(H^s_{r_1})\cap L^{q_2}_{T}(H^s_{r_2}): \|u\|_{L^{\infty}_TH^s}+\|u\|_{L^{q_1}_TH_{r_1}^s}+\|u\|_{L^{q_2}_TH_{r_2}^s} \leq \rho  \right\}$$
$$d(u,v):= \|u-v\|_{L^{\infty}_T H^s \cap L^{q_1}_T H^s_{r_1} \cap L^{q_2}_T H^s_{r_2}}.$$
At this point, using Lemma \ref{mE}, Lemma \ref{st}, Hardy-Littlewood-Sobolev inequality and H\"older's inequality and following closely the proof of Theorem \ref{LW}, we are able to that $J$ maps $Y^s_{T, \rho}$ to itself and $J$ is contraction on $Y^s_{T, \rho}$. This completes the proof.
\end{proof}


\begin{proof}[Proof of Theorem \ref{sc}]
Let  $(Z^s_{\rho, R}, d)$ be a complete metric space with metric $d$ defined by
\[ Z^{s}_{\rho, R}:=\left \{ u \in  L^4(\R, H^s_{\frac{2N}{N-1}}): \|u\|_{L^4(\R, \dot{H}_{\frac{2N}{N-1}}^{s_{\gamma_1}})} \leq \rho,  \|u\|_{L^4(\R, \dot{H}_{\frac{2N}{N-1}}^{s})} \leq R \right\}, \]
\begin{eqnarray*}
d(u,v):=\|u-v\|_{L^4(L^{\frac{2N}{2N-1}})}.
\end{eqnarray*}
From Lemmas \ref{st} and \ref{mE},  we derive that
\begin{eqnarray*}
\|J(u)\|_{L^4(\R, \dot{H}_{\frac{2N}{N-1}}^{s_{\gamma_i}})}  & \lesssim &  \|u_0\|_{\dot{H}^{s_{\gamma_i}}}} + \|F(u)\|_{L^{\frac{4}{3}}(\dot{H}_{\frac{2N}{N+1}}^{s_{\gamma_i}})\\
& \lesssim &\|u_0\|_{\dot{H}^{s_{\gamma_i}}}} +  \sum_{i \in \{1,2\}} \|u\|^2_{L^4(L^{\frac{2N}{N-\gamma_i+1}})}    \|u\|_{L^{4}(\dot{H}_{\frac{2N}{N-1}}^{s_{\gamma_i}})\\
& \lesssim  &\|u_0\|_{\dot{H}^{s_{\gamma_i}}}} +   \sum_{i \in \{1,2\}}   \|u\|^3_{L^{4}(\dot{H}_{\frac{2N}{N-1}}^{s_{\gamma_i}})
\end{eqnarray*}
and
\begin{eqnarray*}
\|J(u)\|_{L^4(\R, \dot{H}_{\frac{2N}{N-1}}^{s})}  & \lesssim &  \|u_0\|_{\dot{H}^{s}}} + \sum_{i \in \{1,2\}}    \|u\|^2_{L^{4}(\dot{H}_{\frac{2N}{N-1}}^{s_{\gamma_i}})} \|u\|_{L^{4}(\dot{H}_{\frac{2N}{N-1}}^{s}).
\end{eqnarray*}
Similarly, we can show that
\begin{align*}
\|J(u)-J(v)\|_{L^4(\R, L^{\frac{2n}{n-1}})} & \lesssim \|F(u)-F(v)\|_{L^{\frac{4}{3}}(L^{\frac{2N}{N+1}})}\\
& \lesssim \sum_{i\in \{1,2 \}}\left( \|u\|^2_{L^4 (L^{\frac{2N}{N-\gamma_i+1}})}+\|v\|^2_{L^4 (L^{\frac{2N}{N-\gamma_i+1}})} \right) \|u-v\|_{L^4(L^{\frac{2N}{N-1}})}\\
& \lesssim \sum_{i\in \{1,2 \}}\left( \|u\|^2_{L^4(\dot{H}^{s_{\gamma_i}}_{\frac{2N}{N-1}})}+\|v\|^2_{L^4 (\dot{H}^{s_{\gamma_i}}_{\frac{2N}{N-1}})} \right) d(u,v).
\end{align*}
If we choose $R$ and $\rho$ sufficiently small $\rho$ such that
\[ C \|u_0\|_{H^{s_{\gamma_i}}}\leq \frac{\rho}{2},  \quad C \|u_0\|_{\dot{H}^s} \leq \frac{\rho}{2},  \quad 2C \rho^2 \leq \frac{1}{4}\]
then $J$ maps $Z^s_{\rho, R}$ to itself and is a contraction map.  By Lemma \ref{st}, it then follows that $u \in (C\cap L^{\infty})(\R, H^s(\R^N)).$ To prove the  scattering,  we define a function $u^+$ by
\[u^+= u_0-\textnormal{i} \int_0^{\infty} U(-s)F(u)(s) \, ds. \]
Since the solution $u\in Z^{s}_{\rho, R}, u^+ \in H^s(\R^N).$  Therefore,  we obtain that
\begin{align*}
\|u(t)-u^+(t)\|_{H^s} & \lesssim  \sum_{i \in \{1,2\}} \|u\|^2_{L^4(t, \infty); \dot{H}_{\frac{2N}{N-1}}^{s_{\gamma_i}})} \|u\|_{L^4(t, \infty; H^{s}_{\frac{2N}{N-1}})} \to 0 \quad  \text {as} \quad t\to \infty.
\end{align*}
This completes the proof.
\end{proof}

\begin{proof} [Proof of Theorem \ref{miD}]
Let us introduce the space
\begin{align*}
Y(T)  & = \left\{\phi \in C\left([0,T], L^2(\mathbb R^N) \right): \|\phi \|_{L^{\infty}([0, T], L^2)}  \leq 2 \|u_0\|_{L^2}, \|\phi\|_{L^{\frac{8}{\gamma_i}} ([0,T], L^{\frac{4N}{2N-\gamma_i}}  ) } \lesssim \|u_0\|_{L^2}, i=1,2\right\}
\end{align*}
and the distance
$$d(\phi_1, \phi_2)= \max\{ \|\phi_1 - \phi_2 \|_{L^{\frac{8}{\gamma_i} }\left( [0, T], L^{\frac{4d}{(2d- \gamma_i)}}\right)}: i=1,2 \}.$$ Then $(Y, d)$ is a complete metric space. Now we show that $\Phi$ takes $Y(T)$ to $Y(T)$ for some $T>0.$
We put
$$ \ q_i= \frac{8}{\gamma_i}, \  r_i= \frac{4N}{2N- \gamma_i}.$$
Note that $(q,r)$ is admissible and
$$ \frac{1}{q'_i}= \frac{4- \gamma_i}{4} + \frac{1}{q_i}, \quad  \frac{1}{r'_i}= \frac{\gamma_i}{2N} + \frac{1}{r_i}.$$
Let $(\bar{q}, \bar{r}) \in \{ (q_i,r_i), (\infty, 2) \}.$ By Theorem \ref{st} and  H\"older inequality, we have that
\begin{align*}
\|J(u)\|_{L_{t,x}^{\bar{q}, \bar{r}}} &  \lesssim   \|u_0\|_{L^2} + \sum_{i \in \{ 1,2 \}}  \|( |\cdot|^{-\gamma_i} \ast |u|^2)u \|_{L_{t,x}^{q'_i,r'_i}}\\
& \lesssim   \|u_0\|_{L^2} + \sum_{i \in \{ 1,2 \}} \| |\cdot|^{-\gamma_i} \ast |u|^2\|_{L_{t,x}^{\frac{4}{4-\gamma_i}, \frac{2N}{\gamma_i}}}
\|u\|_{L^{q_i,r_i}_{t,x}}.
\end{align*}
Since $0<\gamma_i< \min \{2, N \}$, by Theorem \ref{st},  then
\begin{align*}
\| |\cdot|^{-\gamma_i} \ast |u|^2\|_{L_{t,x}^{\frac{4}{4-\gamma_i}, \frac{2N}{\gamma_i}}}  & =   \left\|  \| |\cdot |^{-\gamma_i} \ast |u|^2\|_{L_x^{\frac{2N}{\gamma_i}}} \right\|_{L_t^{\frac{4}{4- \gamma_i}}}\\
& \lesssim    \left \| \||u|^2\|_{L_x^{\frac{2N}{2N- \gamma_i}}} \right\|_{L_t^{\frac{4}{4- \gamma_i}}} \lesssim\|u\|^2_{L_{t,x}^{\frac{8}{4- \gamma_i},r}}\lesssim  T^{1- \frac{\gamma_i}{2}} \|u\|^2_{L_{t,x}^{q_i,r_i}},
\end{align*}
where for the last inequality we have used inclusion relation for the $L^p$ spaces on finite measure spaces: $\|\cdot\|_{L^p(X)} \leq \mu(X)^{\frac{1}{p}-\frac{1}{q}} \|\cdot \|_{L^{q}(X)}$ if measure of $X=[0,T]$ is finite, and $ 0<p<q<\infty$.) Therefore, we conclude that
$$
\|J(u)\|_{L_{t,x}^{\bar{q}, \bar{r}}}   \lesssim   \|u_0\|_{L^2}+T^{1- \frac{\gamma_i}{2}} \|u\|^3_{L_{t,x}^{q_i,r_i}}.$$
This shows that $\Phi$ maps $Y(T)$ to $Y(T).$  Next, we show $J$
is a contraction. For this, as calculations performed  before, first we note that
\begin{align}\label{mi}
 \|( |\cdot |^{-\gamma_i} \ast |v|^{2})(v-w)\|_{L_{t,x}^{q',r'}} \lesssim  T^{1-\frac{\gamma}{2s_i}} \|v\|^2_{L_{t,x}^{q,r}} \|v-w\|_{L_{t,x}^{q,r}}.
\end{align}
Put $\delta_i = \frac{8}{4-\gamma_i}.$  Notice that $\frac{1}{q'_i}= \frac{1}{2}+ \frac{1}{\delta_i}, \frac{1}{2}= \frac{1}{\delta_i} + \frac{1}{q_i}$. In view of H\"older inequality, then
\begin{align}\label{mi1}
 \|( |\cdot|^{-\gamma_i} \ast (|v|^{2}- |w|^{2}))w\|_{L_{t,x}^{q'_i,r'_i}} & \lesssim  \| |\cdot |^{-\gamma_i} \ast \left( |v|^2-|w|^2\right)\|_{L_{t,x}^{2, \frac{2N}{\gamma_i}}} \|w\|_{L_{t,x}^{\delta_i, r_i}} \nonumber \\
 & \lesssim  ( \| |\cdot |^{-\gamma_i} \ast (v (\bar{v}- \bar{w})) \|_{L_{t,x}^{2, \frac{2N}{\gamma}}} +
\| |\cdot |^{-\gamma_i} \ast \bar{w}(v-w)) \|_{L_{t,x}^{2, \frac{2N}{\gamma_i}}} ) \|w\|_{L_{t,x}^{\delta_i,r_i}} \nonumber\\
 & \lesssim \left( \|v\|_{L_{t,x}^{\delta,r}} \|w\|_ {L_{t,x}^{\delta,r}} +\|w\|^2_ {L_{t,x}^{\delta_i,r_i}}\right) \|v-w\|_{L_{t,x}^{q_i,r_i}}\nonumber\\
 & \lesssim  T^{1-\frac{\gamma_i}{2}}  \left( \|v\|_{L_{t,x}^{q_i,r_i}} \|w\|_ {L_{t,x}^{q_i,r_i}} +\|w\|^2_ {L_{t,x}^{q_i,r_i}}\right)  \|v-w\|_{L_{t,x}^{q_i,r_i}}.
\end{align}
Using \eqref{mi}  \eqref{mi1} and the identity
$$(K\ast |v|^{2})v- (K\ast |w|^{2})w= (K\ast |v|^{2})(v-w) + (K \ast (|v|^{2}- |w|^{2})),$$
we get that
\begin{align*}
\|J(v)- J(w)\|_{L_{t,x}^{q,r}} & \lesssim   \|(K\ast |v|^{2})(v-w)\|_{L_{t,x}^{q',r'}} + \|(K \ast (|v|^{2}- |w|^{2}))w\|_{L_{t,x}^{q',r'}}\\
& \lesssim    T^{1-\frac{\gamma}{2s_i}}  \left( \|v\|^2_{L_{t,x}^{q,r}}  +\|v\|_{L_{t,x}^{q,r}} \|w\|_ {L_{t,x}^{q,r}} +\|w\|^2_ {L_{t,x}^{q,r}}\right) \|v-w\|_{L_{t,x}^{q,r}}.
\end{align*}
This gives that $J$ is a contraction form $Y(T)$ to $Y(T)$  provided that $T$ is sufficiently small. Then there exists a unique $u \in Y(T)$ solving \eqref{nlscn}. The global existence of the solution \eqref{nlscn} follows from the conservation of the $L^2$-norm of $u.$ The last property of the proposition then follows from the Strichartz estimates applied with an arbitrary $s_i-$fractional admissible pair on the left hand side and the same pairs as above on the right hand side. The conservation of mass can be attained in a standard way. This completes the proof.
\end{proof}

To prove Theorem \ref{mtbup}, we first introduce
\[ V_a(t)= \int_{\R^N} a(x) |u(t,x)|^2 dx.\]
For a solution $u$ satisfying
\[i\partial_t u + \Delta u = \mathcal{N}.\]
We note that
\[V_a'(t)= 2 \  \text{Im} \int_{\R^N} \nabla a(x)\cdot  \nabla u(t,x) \bar{u}(t,x) \, dx \]
and
\begin{align*}
V_a''(t) & = 4 \text{Re} \sum_{1\leq j, k \leq N} \int_{\R^N} \partial_{jk} a(x) \bar{u}_j(t,x) u_k(t,x) \,dx\\
&\quad +\int_{\R^N} (-\Delta \Delta) a(x) |u(t,x)|^2 dx + 2 \int_{\R^N} \{ \mathcal{N}, u \}_p (t,x) \nabla a(x) \,dx
\end{align*}
where $\{ f, g \}$ is the momentum bracket  defined as $\text{Re} (f \nabla \bar{g}- g \nabla \bar{f}).$
Taking $\mathcal{N}=\left(|x|^{-\gamma_1} \ast |u|^2\right) u - \left(|x|^{-\gamma_2} \ast |u|^2\right) u,$  we obtain the following result. 

\begin{lem} For any $a\in C^4(\R^N),$ we have that
\[V_a'(t)= 2 \  \text{Im} \int_{\R^N} \nabla a(x)\cdot  \nabla u(t,x) \bar{u}(t,x) dx, \]
and
\begin{align*}
V_a''(t) &=  4 \text{Re} \sum_{1\leq j, k \leq N} \int_{\R^N} \partial_{jk} a(x) \bar{u}_j(t,x) u_k(t,x) dx -\int_{\R^N} \Delta^2 a(x) |u(t,x)|^2 dx \\
& \quad + 2\gamma_1 \int_{\R^N} \int_{\R^N}(x-y)\cdot \nabla a(x) \frac{|u(t,x)|^2 |u(t,y)|^2}{|x-y|^{\gamma_1+2}} dx dy \\
&\quad - 2\gamma_2 \int_{\R^N} \int_{\R^N}(x-y)\cdot \nabla a(x) \frac{|u(t,x)|^2 |u(t,y)|^2}{|x-y|^{\gamma_2+2}} dx dy
\end{align*}
In particular,  if $a$ is radial,  then setting $r=|x|$, we obtain that
\begin{align}
V_a'(t)= 2 \text{Im}  \int_{\R^N} a' \frac{x\cdot \nabla u}{r} \bar{u} dx,
\end{align}
\begin{align}\label{dc}
\begin{split}
V_a''(t) & =   4\int_{\R^N} \frac{a'(r)}{r} |\nabla u|^2 dx +4 \int_{\R^N} \left( \frac{a''(r)}{r^2} - \frac{a'(r)}{r^3} \right) |x\cdot u|^2 dx- \int_{\R^N} \Delta^2 a |u|^2 dx  \\
&  \quad + \gamma_1 \int_{\R^N} \int_{\R^N}(x-y)\cdot [\nabla a(x)- \nabla a(y)] \frac{|u(t,x)|^2 |u(t,y)|^2}{|x-y|^{\gamma_1+2}} dx dy  \\
& \quad - \gamma_2 \int_{\R^N} \int_{\R^N}(x-y)\cdot [\nabla a(x)- \nabla a(y)] \frac{|u(t,x)|^2 |u(t,y)|^2}{|x-y|^{\gamma_2+2}} dx dy. 
\end{split}
\end{align}

\end{lem}

\begin{lem}\label{tl1} For any $\eta_0>0 $ and $t\leq \eta_0 R/ (2m_0C_0),$ we have that
\begin{align}
\int_{|x|\geq R} |u(t,x)|^2 dx \leq \eta_0 +o_R(1),
\end{align}
where $m_0=\|u_0\|_{L^2}$ and $o_R(1)\to 0$ as $R \to \infty.$
\end{lem}
Inspired by the virial identity
\[ \frac{d^2}{dt^2} \int_{\R^N} |x|^2 |u(t)|^2 dx= 16 K(u(t)),\]
where
\begin{eqnarray*}
K(u(t)) = \frac{1}{2} \int_{\R^N} |\nabla u|^2 dx + \frac{\gamma_1}{8} \int_{\R^N} \int_{\R^N} \frac{|u(t,x)|^2 |u(t,y)|^2}{|x-y|^{\gamma_1}} dx dy - \frac{\gamma_2}{8} \int_{\R^N} \int_{\R^N} \frac{|u(t,x)|^2 |u(t,y)|^2}{|x-y|^{\gamma_2}} dx dy.
\end{eqnarray*}
We can rewrite $V''_a(t)$ in \eqref{dc} as
\begin{align}
V''_a(t)=16 K(u(t)) + R_1 +R_2+R_3 + R_4,
\end{align}
where
\[R_1:= 4\int_{\R^N} \frac{a'(r)}{r} |\nabla u|^2 dx +4 \int_{\R^N} \left( \frac{a''(r)}{r^2} - \frac{a'(r)}{r^3} \right) |x\cdot u|^2 dx, \]
\[  R_2:=\gamma_1 \int_{\R^N} \int_{\R^N}[(x-y)\cdot (\nabla a(x)- \nabla a(y))- 2|x-y|^2 ] \frac{|u(t,x)|^2 |u(t,y)|^2}{|x-y|^{\gamma_1+2}} dx dy,\]
\[  R_3:=-\gamma_2 \int_{\R^N} \int_{\R^N}[(x-y)\cdot (\nabla a(x)- \nabla a(y))- 2|x-y|^2 ] \frac{|u(t,x)|^2 |u(t,y)|^2}{|x-y|^{\gamma_2+2}} dx dy,\]
\begin{align*}
R_4 := - \int_{\R^N} \Delta^2u |u|^2 dx= - \int_{\R^N} \left( a^{(4)} + \frac{2 (d-1)}{r} a^{(3)} + \frac{(d-1) (d-3)}{r^2} \left( a''- \frac{a'}{r} \right) \right) |u|^2 dx.
\end{align*}
We will show that $R_1, R_2$ and $R_3$ are error terms,  by making use of localization.  We take $a\in C^4(R^N)$ radial satisfying
$$a''(r)=  \begin{cases} 2, \quad 0 \leq r \leq 2R\\
0, \quad r \geq 4R,
\end{cases} $$
and $a(0)=a'(0)=0, a''\leq 2$ and $a^{(4)} \leq R^{-2}.$  Then we have the following result.

\begin{lem} \label{tl2} There exists a constant  $C_1=C_1(s, \gamma, N, m_0)$ such that
\begin{align}
V_a''(t) \leq  16 K(u(t)) + C_1 \|u\|^{2-\gamma_1/2}_{L^2(|x| \geq 2R)} + C_1 \|u\|^{2-\gamma_2/2}_{L^2(|x| \geq 2R)}.
\end{align}
\end{lem}
\begin{proof}
We first prove that $R_1\leq 0.$ Indeed,  if $a''(r)/r^2-a'(r)/r^3 \leq 0,$ then $R_1 \leq 0$ since $|a'| \leq 2r.$ Also,  if $a''(r)/r^2-a'(r)/r^3 \geq 0,$ then
\[R_1 \leq 4 \int_{\R^N} (a''-2) |\nabla u|^2 dx  \leq 0. \]
Moreover,  we know that
\begin{eqnarray*}
\text{supp} \{ (x-y)\cdot (\nabla a(x)- \nabla a(y))-2 |x-y|^2 \} \subset  \cup_{z\in \{ x, y \}} \{ (x,y): |z| \geq 2R \}.
\end{eqnarray*}
In the region where $|x| \geq 2R,$
\[|(x-y)\cdot \left( \nabla a (x) - \nabla a(y) \right) | \lesssim |x-y|^2.  \]
By  H\"older's inequality, Young's inequality and Sobolev embedding theorem, we can control $R_2$ and $R_3$ from this region:
\begin{align*}
\int_{\R^N} \int_{|x| \geq 2R} \frac{|u(t,x)|^2 |u(t,y)|^2}{|x-y|^{\gamma_i}} dx dy & \lesssim \|u\|_{L^{2N/ (N-\gamma_i/2)} (|x|\geq 2 R)}^{2} \|u\|^2_{L^{\frac{2N}{N-\gamma_i/2}}}\\
& \lesssim  \|u\|_{L^2(|x|\geq 2R)}^{2- \gamma_i/2} \|u\|^{2- \gamma/2}_{L^2}  \|u\|^{\gamma_i/2}_{\dot{H}^1}\\
& \lesssim \|u\|^{2-\gamma_i/2}_{L^2(|x|\geq 2R)}
\end{align*}
Similarly,   we have the same control in the region where $|y| \geq 2R.$  Thus, we derive that
\[R_2 \lesssim \|u \|^{2-\gamma_1/2}_{L^2(|x|\geq 2R)}, \quad R_3 \lesssim \|u \|^{2-\gamma_2/2}_{L^2(|x|\geq 2R)}.\]
Furthermore,  we have that
\[R_4 \lesssim R^{-2} \|u\|^2_{L^2(|x|\geq 2R)}. \]
Hence the proof is completed.
\end{proof}
\begin{proof}[Proof of Theorem  \ref{mtbup}] The first part on finite blow-up is standard,  see e.g.  \cite{TC}.   We shall only prove second part on infinite time blow-up.  If assume that
\begin{align}
\sup_{t\in \R^+} \|u(t)\|_{H^1} \leq C_0< \infty.
\end{align}
In light of Lemmas \ref{tl1} and \ref{tl2},  then
\begin{align}
V''_a(t) \leq 16 K(u(t)) + C_1 \left(\eta_0^{2- \gamma_1/2} + \eta_0^{2- \gamma_2/2}  + o_R(1) \right)
\end{align}
for any $t\leq T= \eta_0 R/ (2m_0 C_0).$ Integrating the above inequality from 0 to $T$ twice and using the fact that
\[K(u(t)) \leq E(u(t))<0  \quad \text{for any } \ t \in \R,\]
we have that
\begin{align*}
V_a(T) &\leq  V_a(0) + V_a'(0)T + \int_0^t \int_0^t g (16 K(u(s)) + C_1(\eta_0^{2- \gamma_1/2} + \eta_0^{2- \gamma_2/2}  + o_R(1)) ) ds dt\\
& \leq  V_a(0) + V_a'(0)T + \left( 16 E(u_0) + C_1(\eta_0^{2- \gamma_1/2} + \eta_0^{2- \gamma_2/2}  + o_R(1))  \right) T^2.
\end{align*}
Taking $\eta_0$ such that $C_1(\eta_0^{2- \gamma_1/2} + \eta_0^{2- \gamma_2/2})= -\frac{1}{2} E(u_0),$ and $R$ large enough, one has that
\begin{align}\label{fc1}
V_a(T) \leq V_a(0) + V_a'(0) \frac{\eta_0 R}{2m_0 C_0} + \alpha_0 R^2,
\end{align}
where $\alpha_0= E(u_0) \eta_0^2/ (4m_0C_0)^2<0.$ Next we claim that
\begin{align}\label{fc2}
V_a(0)=o_R(1)R^2, \quad V_a'(0)=o_R(1)R.
\end{align}
In fact, note that
\begin{align*}
V_a(0) & \leq   \int_{|x| \leq \sqrt{R}} |x|^2|u_0(x)|^2 dx + \int_{\sqrt{R} \leq |x| \leq 2R} |x|^2 |u_0(x)|^2 dx + R^2 \int_{|x|\geq 2R} |u_0(x)|^2 dx\\
& \leq  Rm_0^2+ R^2 \int_{|x| \geq \sqrt{R}} |u_0(x)|^2 dx + R^2 \int_{|x|\geq 2R} |u_0(x)|^2 dx\\
& =  o_R(1) R^2.
\end{align*}
Similarly,  we  have that $V_a'(0)=o_R(1)R^2.$
Taking $R>0$ large enough,  by \eqref{fc1} and \eqref{fc2},  we infer that
\[V_a(T) \leq o_R(1)R^2 + \alpha_0R^2 \leq \frac{1}{2} R^2<0. \]
This contradicts $V_a(T)\geq 0.$  This completes the proof.
\end{proof}

\end{document}